\documentclass{amsart}
\usepackage[latin1]{inputenc}
\usepackage[T1]{fontenc}
\usepackage[all]{xy}
\usepackage{amsmath}
\usepackage{amssymb}
\usepackage{latexsym}
\usepackage{enumerate}
\usepackage[active]{srcltx}
\usepackage{graphics}
\usepackage{epsfig}
\usepackage{tikz-cd}
\usepackage{mathdots}

\DeclareMathOperator{\Sym}{Sym}
\DeclareMathOperator{\Hom}{Hom}
\DeclareMathOperator{\Mas}{Mas}

\newtheorem{theorem}{Theorem}[section]
\newtheorem{propdef}{Definition-Proposition}[section]
\newtheorem{proposition}[theorem]{Proposition}
\newtheorem{corollary}[theorem]{Corollary}
\newtheorem{lemma}[theorem]{Lemma}

\newtheorem{remark}[theorem]{Remark}

\newtheorem{definition}[theorem]{Definition}

\newcommand{\Z}{{\mathbb Z}}

\newcommand{\cS}{\mathcal{S}}

\newcommand{\downarrowright}[1]{\downarrow
\rlap{\raise0.1cm\hbox{$\scriptstyle{#1}$}}}

\newcommand{\downarrowleft}[1]{\rlap{\kern-0.2cm
\raise0.1cm\hbox{$\scriptstyle{#1}$}}\downarrow}

\newcommand{\uparrowright}[1]{\uparrow
\rlap{\lower0.1cm\hbox{$\scriptstyle{#1}$}}}

\newcommand{\uparrowleft}[1]{\rlap{\kern-0.2cm
\lower0.1cm\hbox{$\scriptstyle{#1}$}}\uparrow}

\def\umono{\ar@{_{(}->}[u]}
\def\uumono{\ar@{_{(}->}[uu]}

\def\lmono{\ar@{_{(}->}[l]}
\def\llmono{\ar@{_{(}->}[ll]}

\makeatletter
\@namedef{subjclassname@2020}{\textup{2020} Mathematics Subject Classification}
\makeatother

\begin{document}

\title{Some remarks on the Maslov index}

\author{Wolfgang Pitsch}
\address{Departament de Matemàtiques \\
Universitat Autònoma de Barcelona \\
08193 Bellaterra (Cerdanyola del Vallés) }
\email{pitsch@mat.uab.es}

\thanks{Author supported by FEDER/MEC grant ``Homotopy theory of combinatorial and algebraic structures'' PID2020-116481GB-I00}
\date{}




\begin{abstract}
	It is  a classical fact that Wall's index of a triplet of Lagrangians in a symplectic space over a field $k$ defines a $2$-cocycle $\mu_W$ on the associated symplectic group with values in the Witt group of $k$. Moreover, modulo the square of the fundamental ideal this is a trivial $2$-cocycle. In this work we revisit this fact from the viewpoint of the theory of Sturm sequences and Sylvester matrices developed by  J.~Barge and  J.~Lannes in~\cite{MR2431916}. We define a refinement by a factor of $2$ of Wall's cocycle and  use the technology of Sylvester matrices to  give an explicit formula for the coboundary associated to the mod $I^2$ reduction of the cocycle which is valid for any field of characteristic different from $2$. Finally we explicitly compute the values of the coboundary on standard elements of the symplectic group
\end{abstract}

\maketitle

\section{Introduction}

Let $k$ be a field of characteristic different from $2$. Let $W(k)$ be the Witt ring over $k$ and let  $I = \ker (W(k) \rightarrow \Z/2\Z)$ be its fundamental ideal. Denote by  $Sp_{2g}(k)$ the symplectic group of $k$, i.e. the orthogonal group associated to the symplectic bilinear form on  $k^{2g}$, which is given by the $2g \times 2g$  matrix  $\left(  \begin{matrix} 0_g & I_g \\ -I_g & 0_g \end{matrix} \right)$.   Recall that a \emph{Lagrangian}  in $k^{2g}$ is a totally isotropic subspace of maximal rank for the symplectic form. In \cite{MR0246311} Wall associates to any three Lagrangians  $L_1, L_2, L_3$ in $k^{2g}$ an element   $\mu_W(L_1,L_2,L_3)  \in W(k)$, nowadays known as Wall's tertiary index. Geometrically this index measures the failure of additivity of the signature of $4k$-manifolds under the operation of gluing along parts of the boundary.    Algebraically, if we fix a Lagrangian say L, we have then an associated function  $ \mu_W: Sp_{2g}(k) \times Sp_{2g}(k) \rightarrow W(k)$ by the rule $(A,B) \longmapsto \mu_W(L,AL,ABL)$ which is by direct inspection a $2$-cocycle, \emph{Maslov's cocycle}.

  It is then a well-known (see for instance \cite{MR1740881}) that there exists a unique function, 
 $\Phi: Sp_{2g}(k) \rightarrow W(k)/I^2$ such that,
 \[
 \forall A,B \in Sp_{2g}(k) \quad \mu_W(L,AL,ABL) = \Phi(AB) - \Phi(A) - \Phi(B) \textrm{ mod } I^2.
 \]

This statement amounts to saying that the  Maslov cocycle modulo $I^2$ is trivial in a unique way. The unicity part of this statement is very classic. Indeed, any two  trivializations of $\mu_W$ differ by a group homomorphism $Sp_{2g}(k) \rightarrow W(k)/I^2$, and the  unicity  statement is an elementary consequence from the fact that $Sp_{2g}(k)$is perfect, unless  $g=2$ and $k=\mathbb{F}_3$. In this singular case we have     $Sp_2(\mathbb{F}_3) = SL_2(\mathbb{F}_3)$, the abelianization of this group is $\mathbb{Z}/3$ and is induced by the exceptional map $PSL_2(\mathbb{F}_3) \simeq \mathfrak{A}_4 \rightarrow \mathbb{Z}/3$ \cite[p. 78]{MR1357169}. Bu as $W(\mathbb{F}_3) \simeq W(\mathbb{F}_3)/I^2 \simeq \mathbb{Z}/4 $ \cite[Lemma 1.5 p.87]{MR0506372}, there are no non-trivial homomorphisms $Sp_2(\mathbb{F}_3) \rightarrow W(\mathbb{F}_3)/I^2)$ and unicity follows.

The aim of this work is to revisit this results inside the framework introduced in \cite{MR2431916}. Our first task, in Section~\ref{sec:chemlagr} will be to associate to any \emph{Lagrangian path}, i.e. any finite sequence $\alpha: \Lambda_0,\Lambda_1,\dots,\Lambda_n$ of Lagrangians such that any two consecutive terms are transverse, a bilinear form called its \emph{Sylvester matrix} $S(\alpha)$. We will study the behavior of the Sylvester matrix under the operation  of path concatenation. The bulk of Section~\ref{sec:chemlagr} is the proof of the

\begin{lemma}[Shortcut Lemma]
	Let $\Lambda_0, \cdots, \Lambda_{n+1}$  be a Lagrangian path, and let  $S$ denote its Sylvester matrix. Assume that there exist two indices  $0 \leq i < j \leq n+1$ such that $\Lambda_i \pitchfork \Lambda_j$.  Then we have two new Lagrangian paths: 
	\begin{enumerate}
		\item[a)] The sub-sequence $\Lambda_i \cdots, \Lambda_j$, whose Sylvester matrix is $S(\Lambda_i,\dots,\Lambda_j)$,
		\item[b)] the shortened sequence $\Lambda_0, \cdots, \Lambda_i,\Lambda_j, \cdots \Lambda_{n+1}$, whose Sylvester matrix is
		
		$S(\Lambda_0,\dots,\Lambda_i,\Lambda_j,\dots,\Lambda_{n+1})$.
	\end{enumerate}
	Then,
	\[
	S \textrm{ is isometric to }S(\Lambda_i,\dots,\Lambda_j) \bot S(\Lambda_0,\dots,\Lambda_i,\Lambda_j,\dots,\Lambda_{n+1}),
	\]
	where $\bot$ stands for the orthogonal sum of bilinear forms.
\end{lemma}

In Section~\ref{sec:Maslovindex} we use the Shortcut Lemma to prove that  that given three Lagrangians $\Lambda_0,\Lambda_1,\Lambda_2$, if we choose two Lagrangian paths $\alpha_{01}$ and $\alpha_{12}$ respectively starting at $\Lambda_0$ and $\Lambda_1$ and ending at $\lambda_1$ and$\lambda_2$ then the class in $W(k)$ of the bilinear form
\[
\mu_{BL}(\Lambda_0,\Lambda_1,\Lambda_2) = S(\alpha_{01} \ast \alpha_{12}) \bot - S(\alpha_{01})\bot-S(\alpha_{12})
\] 
is independent of the choice of paths. We then show that the index $\mu_{BL}$ satisfies the characteristic properties of the Maslov index, and compare it to the Wall-Kashiwara index.

\begin{theorem} The index $\mu_{BL}$ satisfies the following properties:
	\begin{enumerate}
		\item If any two of the three Lagrangians $\Lambda_0,\Lambda_1,\Lambda_2$ are equal, then \\ $ \mu_{BL}(\Lambda_0,\Lambda_1,\Lambda_2)=0$.
		\item If $\phi \in Sp_{2g}(k)$, then $\mu_{BL}(\Lambda_0,\Lambda_1,\Lambda_2) =  \mu_{BL}(\phi\cdot\Lambda_0,\phi\cdot\Lambda_1,\phi\cdot\Lambda_2)$.
		\item The index $\mu$ is a $2$-cocycle, if $\Lambda_0,\Lambda_1,\Lambda_2,\Lambda_3$ are $4$ Lagrangians, then
		\[
		\mu_{BL}(\Lambda_1,\Lambda_2,\Lambda_3) -  \mu_{BL}(\Lambda_0,\Lambda_2,\Lambda_3) +  \mu_{BL}(\Lambda_0,\Lambda_1,\Lambda_3) - \mu_{BL}(\Lambda_0,\Lambda_1,\Lambda_2)=0.
		\]
	\item If $\sigma \in \mathfrak{S}_3$ is a permutation of the indices $0,1$ and $2$, then:
	\[
	\mu_{BL}(\Lambda_0,\Lambda_1,\Lambda_2) =\varepsilon(\sigma) \mu_{BL}(\Lambda_\sigma(0),\Lambda_\sigma(1),\Lambda_\sigma(2)).
	\]
	\item If $\mu_{KW}$ denotes Wall-Kashimara's index of three Lagrangians, then in $W(k)$.
	\[
	2\mu_{BL}(\Lambda_0,\Lambda_1,\Lambda_2) = \mu_{WK}(\Lambda_1,\Lambda_2,\Lambda_3)
	\]
	\end{enumerate}
\end{theorem}

 Finally, in Section~\ref{sec:trivmodI2} proceed to study the associated $2$-cocycle on the symplectic group, given by fixing a Lagrangian $L$ and defining for any to $A,B \in Sp_{2g}(k)$, $\mu_L(A,B)= \mu_{BL}(L,AL,ABL)$. As we have fixed $L$,  the symplectic space is isometric to $L \oplus L^\ast$ with the standard symplectic form. Let $\mathcal{S}_L \subseteq Sp_{2g}(k)$ (resp. $\mathcal{S}_{L^\ast}$) the stabilizer of $L$ (resp. $L^\ast$) and consider the canonical evaluation amp $E: \mathcal{S}_L \ast \mathcal{S}_{L^\ast} \rightarrow Sp_{2g}(k)$ from the free product of the two stabilizers to the symplectic group. Almost by definition, an element in $\mathcal{S}_L \ast \mathcal{S}_{L^\ast}$ is  sequence of quadratic forms $q_n,q_{n+1}, \dots, q_m$ alternatively defined on $L$ and on $L^\ast$, such a sequence is called a Sturm sequence in \cite{MR2431916} and Barge and Lannes show there how to associate to a Sturm sequence a Lagrangian path and hence a Sylvester matrix.

The technology of Sturm sequences and Sylvester matrices allows to define $4$ canonical functions $f_{00},f_{01},f_{11},f_{10}$ on the free product $\mathcal{S}_L\ast\mathcal{S}_{K^\ast}$. The behavior of these functions with respect to the free product shows that there is a commutative diagram:
\[
\xymatrix{
	1 \ar[r] & K \ar[r] \ar[d]^-{f_{01}} & \mathcal{S}_L \ast \mathcal{S}_{L^\ast} \ar[r]  \ar[dl]^{f_{00}} \ar[d] & Sp_{2g}(k) \ar@{=}[d] \ar[r] & 1 \\
	0 \ar[r] & W(k) \ar[r]& \Gamma \ar[r] & Sp_{2g}(k) \ar[r] & 1
}
\]
where$f_{00}$, restricted to the kernel $K$, is a homomorphism but the retraction  $f_{01}$ \emph{is not}.  By standard group cohomology arguments, $f_{01}$ defines a $2$-cocycle for the bottom extension, and by construction of the functions this is in fact $\mu_L$. Finally, by analyzing the image of $f_{01}$, we show that its reduction mod $I^2$, where $I$ is the fundamental ideal in the Witt group is a homomorphism, therefore $f_{01} \textrm{ mod } I^2$ is our desired trivialization. We finally compute explicitly this function for typical elements in the Symplectic group.

{\bf Acknowledgments: } This work grew out of  series  of visits by Jean Barge to Barcelona around 2005. During these he explained  to me the present point of view of his joint work with Jean Lannes~\cite{MR2431916} and we worked out the details of the proofs of their statements, particularly the Short Cut Lemma, and the explicit splitting of the mod-2 reduction of the Maslov cocycle. His recent passing away prompted me to fully publish what we had worked out together. As a further recognition I dedicate this article to the memory of my former PhD advisor, Professor Jean Barge. 

\section{General background}\label{sec:rappels}

\subsection{Witt Monoïd and Witt Group}\label{subsec:Witt}

We recall here some elementary facts about the Witt group of a field of characteristic different from $2$, for more ample information and in particular for the proofs of the results presented we refer the reader to~\cite{MR0506372} or~\cite{Lam05}.

Let $V$ denote a  $k$-vector space  and denote by  $V^\ast = \Hom_k(V,k)$ its dual. Let then  $\Sym(k)$ denote the set of all symmetric bilinear forms  defined on finite dimensional $k$-vector spaces up to isometry. An element in  $\Sym(k)$ is represented by a pair $(P,q)$, where $P$ is a $k$-vector space of finite dimension and  $q: P \rightarrow  P^\ast$ is a $k$-linear map that coincides with its dual map $q^\ast: P^{\ast \ast} \rightarrow P^\ast$ up to the canonical identification   $P \simeq P^{\ast\ast}$. The space $P$ is by definition the  \emph{support} of $q$, if clear from the context  we will omit the support from the notation and write simply $q$ for $(P,q)$.

Recall that a symmetric bilinear form  $(P,q)$ is \emph{non-degenerate} if and only if $q: P \rightarrow P^\ast $ is an isomorphism, and that a symmetric bilinear form  $(P,q)$  is \emph{neutral} if and only if it is non-degenerate and  there exists  a subvector space $I \subset P$ that coincides with its own orthogonal $I = I ^\bot$.

The orthogonal sum of symmetric bilinear forms, which we denote by $\bot$, endows the set  $\Sym(k)$ with the structure of a commutative monoïd. The isometry classes of neutral forms determine a sub-monoïd $\operatorname{Neut}(k) \subseteq \Sym(k)$, the quotient $MW(k) = \Sym(k)/\operatorname{Neut}(k)$ is by definition the \emph{Witt monoïd} of  $k$. More precisely two symmetric bilinear forms  $(P_1,q_1)$ and $(P_2,q_2)$ in $\Sym(k)$ are  equivalent if and only if there exist two neutral forms $(N_1,n_1)$ and $(N_2,n_2)$ such that $q_1 \bot n_1 \sim q_2 \bot n_2$, where we denote by $X \sim Y$ the fact that $X$ and $Y$ are isometric.

The Witt \emph{group}  $W(k)$ is then the image in  $MW(k)$ of the sub-monoïd of  $\Sym(k)$ generated by the symmetric non-degenerate bilinear forms. Given that for a non-degenerate form $q$ the orthogonal sum  $q \bot -q$  is neutral, we get indeed a group structure: the inverse of $q$ is  $-q$. In addition, the tensor product of bilinear forms endows  $W(k)$ with a commutative multiplication compatible with the orthogonal sum  and endows $W(k)$ with the structure of a unital commutative ring.

By definition there is an injection  $W(k) \rightarrow MW(k)$ and, as  $k$ is a field,  we have a canonical retraction $MW(K) \rightarrow W(k)$ called the  \emph{regularization} map;  it sends the symmetric bilinear  form $(P,q)$  onto the induced form on the quotient $P/P^\bot$, where $P^\bot = \ker q$ is the \emph{radical} of $\alpha$.  For a general symmetric bilinear form  $(P,q)$, its class $\overline{(P,q)} \in W(k)$ will always refer to the class of its regularized form. Finally, given a unit $a \in k^{\times}$ we will denote by $\langle a \rangle$ the bilinear form on $k$ with associated matrix $[a]$.

The ring  $W(k)$ has a unique maximal ideal $I$ such that $W(k)/I \simeq \mathbb{Z}/2$, its \emph{fundamental ideal}, which is the kernel of the map "rank mod $2$": $I = \ker (W(k) \rightarrow \Z/2)$, that sends a bilinear form onto the mod $2$ reduction of the dimension of its support; as neutral forms have even rank this is indeed a well-defined map. It is known that the fundamental ideal  $I$  is generated by the Pfister forms $\langle 1,-\lambda\rangle = \langle 1\rangle \bot \langle -\lambda\rangle$, for $\lambda$ a unit in  $k$,  \cite[p. 316]{Lam05}.

In this work we will be more particularly interested in the quotient $W(k)/I^2$, which is by definition part of an extension of abelian groups
\[
\xymatrix{
	0 \ar[r] & I/I^2 \ar[r] & W(k)/I^2 \ar[r] & \mathbb{Z}/2 \ar[r] & 0. & (\ast)
}
\]

The kernel  $I/I^2$ is isomorphic to the multiplicative group of units in $k$ up to the squares, $I/I^2 = k^{\times}/(k^{\times})^2$ via the discriminant map, $\operatorname{dis}: W(k) \longrightarrow  k^\ast/(k^\ast)^2$ that sends a non-degenerate bilinear form  $q$ of rank $r$ onto $(-1)^{\frac{r(r-1)}{2}}\operatorname{det}(q)$. An elementary but important fact for the present work is that 
this exact sequence does not usually split. For instance, by the Gauss algorithm for reducing quadratic forms,   $W(\mathbb{R} ) \simeq \mathbb{Z}$, then $I = 2\mathbb{Z}$ and in this case $(\ast)$ is:
\[
\xymatrix{
	0 \ar[r] & \mathbb{Z}/2 \ar[r] & \mathbb{Z}/4 \ar[r] & \mathbb{Z}/2 \ar[r] & 0.
}
\]

Even so, for any field $k$  the pull-back of the extension $(\ast)$ along the canonical map $\mathbb{Z}/4 \rightarrow \mathbb{Z}/2$ always splits.

\[
\xymatrix{
	0 \ar[r] & I/I^2 \ar[r] \ar@{=}[d] & \widetilde{W} \ar[r] \ar[d] & \mathbb{Z}/4 \ar[d] \ar[r] \ar@{..>}[dl] & 0 \\
	0 \ar[r] & I/I^2 \ar[r] & W(k)/I^2 \ar[r] & \mathbb{Z}/2 \ar[r] & 0.
}
\]

A section of the pull-back is induced by the morphism  $\mathbb{Z}/4 \rightarrow W(k)/I^2$ that sends $n \in \mathbb{Z}/4$ to the diagonal form  $n\langle 1\rangle$. To check that this map is well-defined, observe that for all integers $m$ the rank  of $4m\langle 1 \rangle$ is even, and its discriminant is  $(-1)^{\frac{4m\times (4m-1)}{2}}=1$.

Combined with the isomorphism  $I/I^2 \simeq k^{\times}/(k^{\times})^2$, the section induces a surjective morphism of groups::
\[
\begin{array}{rcl}
F: k^\times/(k^\times)^2 \oplus \mathbb{Z}/4 & \longrightarrow & W(k)/I^2 \\
(\lambda, n) & \longmapsto & \langle 1,-\lambda\rangle \oplus n\langle 1 \rangle.
\end{array}
\]
We will use this map in our explicit computations of the trivialization of the Maslov cocycle.

\subsection{Lagrangian combinatorics}\label{subsec:comlagr}
Fix an integer $g \geq 1$ and let  $L = k^g$. Denote by $H(L)$ the $k$-vector space $L\oplus L^\ast$ together with the alternating bilinear form $\omega$, known as the \emph{symplectic form},
\[
 \omega((x,\xi),(y,\eta)) = \xi(y) - \eta(x).
\]

The group of isometries of  $\omega$ is by definition the symplectic group $Sp_{2g}(k)$.  Recall that a  \emph{Lagrangian}  in the symplectic space  $H(L)$ is a subvector space   $\Lambda \subset H(L)$ that coincides with its own orthogonal $\Lambda = \Lambda^\bot$.  In the set  $\mathcal{L}$ of  Lagrangians in $H(L)$, we have two canonical elements:  $L$ and $L^\ast$. If $X$ and $X'$  are two subvector spaces in $H(L)$ we will say that  $X$ and $X'$ are \emph{transverse} if and only if $X + X' = H(L)$, we will then write $X \pitchfork X'$. 

The symplectic group has a natural left action on the set $\mathcal{L}$ that is transitive, cf. for instance \cite[\S 2]{Souriau}, the proof for  $k= \mathbb{R}$ is valid for any field. Let $\mathcal{S}_L$ denote the point-wise stabilizer of the Lagrangian $L$; by direct inspection elements in this group  when written by blocks according to the decomposition $H(L)= L \oplus L^\ast$, have the shape
\[
E(q) = 
\begin{pmatrix}
1 & q \\
0 & 1
\end{pmatrix},
\]
where $q: L^\ast \rightarrow L$ satisfies, $q = q^\ast$. In particular $\mathcal{S}_L$ is canonically isomorphic to the $k$-vector space of symmetric bilinear forms with support $L^\ast$. Similarly, $\mathcal{S}_{L^\ast}$, the point-wise stabilizer of $L^\ast$, consists of those matrices of the form

\[
E(q') = 
\begin{pmatrix}
  1 & 0 \\
  q' & 1
\end{pmatrix}.
\]
where $q': L \rightarrow L^\ast$ is symmetric and is isomorphic to the vector space of symmetric bilinear forms with support $L$.

Let  $\Lambda$ be a fixed Lagrangian, and denote by $\mathcal{L}_{\pitchfork \Lambda}$ the set of those Lagrangians that are transverse to  $\Lambda$. Because the canonical action of the symplectic group preserves transversality, we have induced canonical actions  of $\mathcal{S}_L$ on $\mathcal{L}_{\pitchfork L}$ and of $\mathcal{S}_{L^\ast}$ on $\mathcal{L}_{\pitchfork L^\ast}$. These actions are both simple transitive, and hence the bijections
\[
\begin{array}{rclcrcl}
\mathcal{S}_L & \longrightarrow & \mathcal{L}_{\pitchfork L} & &  \mathcal{S}_ {L^\ast} & \longrightarrow & \mathcal{L}_{\pitchfork L^\ast}  \\
\left(\begin{smallmatrix} 
1 & q \\
0 & 1
\end{smallmatrix}
\right) & \longmapsto & \left(\begin{smallmatrix} 
1 & q \\
0 & 1
\end{smallmatrix}
\right) L^\ast & 
   & 
\left(\begin{smallmatrix} 
1 & 0 \\
q & 1
\end{smallmatrix} \right)
 & \mapsto & \left(\begin{smallmatrix} 
1 & 0 \\
q & 1
\end{smallmatrix}
\right) L.
\end{array}
\]
endow the sets $\mathcal{L}_L$ and $\mathcal{L}_{L^\ast}$ of the structure of an affine set over the corresponding stabilizer group. If we fix an arbitrary Lagrangian $\Lambda$, the above discussion leads to two different affine structures on $\mathcal{L}_{\pitchfork \Lambda}$ over the vector space of symmetric bilinear forms with support $\Lambda^\ast$: one where elements in $\mathcal{S}_{\Lambda}$ are written according to the decomposition $H = \Lambda \oplus \Lambda^\ast$, and one where they are written according to the decomposition $H= \Lambda^\ast \oplus \Lambda$. To fix this ambiguity we chose the action corresponding to the decomposition $\Lambda^\ast \oplus \Lambda$. As a consequence, when considering bilinear forms with support $L$, we have to conjugate the described action of $\mathcal{S}_{L^\ast}$  on $\mathcal{L}_{\pitchfork L^\ast}$ by the symplectic matrix $ \left( \begin{smallmatrix}
0 & Id \\
-Id & 0
\end{smallmatrix} \right)$
and this introduces a sign; with this convention the Lagrangian $\left(\begin{smallmatrix} 
1 & 0 \\
q & 1
\end{smallmatrix}
\right) L$ is the translation of $L$ along $-q$ and  $\left(\begin{smallmatrix} 
	1 & p \\
	0 & 1
\end{smallmatrix}
\right) L^\ast$ is the translation of $L^\ast \in \mathcal{L}_{L}$ by $p$.

The affine space structure tells us that given two elements  $M,N \in \mathcal{L}_{\pitchfork \Lambda}$, their difference $M-N= d(N,M)$ is a well-defined symmetric bilinear form with support $\Lambda^\ast$. For any two Lagrangians there is yet another natural map.

\begin{lemma}\label{lem bildeuxlag}
	Let $\Lambda$ and $M$ be two lagangians. The evaluation map
	\[
	\begin{array}{rcl}
	e_{\Lambda,M}: \Lambda \times M & \longrightarrow & k \\
	(\ell,m) & \longmapsto & \omega(\ell,m)
	\end{array}
	\]
	is a bilinear form, it induces by adjunction a linear map
	$\beta_{\Lambda,M} : \Lambda \rightarrow M^\ast$ that satisfies the following properties:
	\begin{enumerate}
		\item $\beta_{\Lambda,M} = -\beta_{M,\Lambda}^\ast$
		\item $\mathrm{Im }\  \beta_{\Lambda,M}= H(L)/(\Lambda+M)$
		\item $\ker \beta_{\lambda,M} = \Lambda \cap M$
		\item $\beta_{\Lambda,M}$  is invertible if and only if $\Lambda \pitchfork M$.
		\item If $ \phi \in Sp_{2g}(k)$, then $  \beta_{\Lambda,M} = (\phi\vert_{M})^\ast \circ \beta_{\phi \Lambda,\phi M} \circ {\phi\vert_\Lambda}$
	\end{enumerate}
	
\end{lemma}
\begin{proof}
 These are all immediate computations.
\end{proof}

The following result explains the links between the morphisms $\beta$ and $d$.
\begin{lemma}\label{lem reldiffbeta}
	Let $X$ be a Lagrangian in $H(L)$, and let $\Lambda,M \in\mathcal{L}_{\pitchfork X}$. The difference of this two Lagrangians is $d_X(\Lambda,M) \in \mathcal{S}_{X^\ast}$ and we have that
	\[
	d_X(\Lambda,M) = - (\beta_{X,M})^{-1}\beta_{\Lambda,M}\beta_{\Lambda,X}^{-1}
	\]
\end{lemma}
\begin{proof} Consider the commutative diagram
	\[
	\xymatrix{
		& M \ar[d]_{i_M} \ar[dr]^{\beta_{M,X}}   \\
		X \ar[r]^-{i_X} & H(L) \ar[r]^-{p_X} & X^\ast. \\
		& \Lambda \ar[u]^{i_\Lambda} \ar[ur]_{\beta_{\Lambda,X}} & \\ 
	}
	\]
	Since $X$ is transverse to both $M$ and $\Lambda$, the morphisms $\beta_{M,X}$ and $\beta_{\Lambda,X}$ are invertible, and this gives us two sections of the middle exact sequence: $\sigma_M = i_M \circ \beta_{M,X}^{-1}$ and $\sigma_\Lambda: i_\Lambda\circ \beta_{\Lambda,X}^{-1}$. Because $\Lambda$ and $M$ are both Lagrangiens, the symplectic form pulled-back along any of the above two sections to $X^\ast$ is the zero form.
	
	Then, by definition, $d_X(\Lambda,M) =  i_M \circ \beta_{M,X}^{-1}-  i_\Lambda\circ \beta_{\Lambda,X}^{-1}: X^\ast \rightarrow X$. More explicitly, if $\alpha \in X^\ast$, there exists a unique couple  $(m,\ell) \in M \times  \Lambda$ such that $\alpha = \omega(m,-) = \omega(\ell,-)$ and $d_X(\Lambda,M)(\alpha) = m-\ell$. By direct computation, th map 
	\[
	\beta_{X,M} d_X(\lambda,M)\beta_{\Lambda,X}: \Lambda \rightarrow M^\ast
	\]
	sends and elemnt $\ell \in \Lambda$ to the linear map that on $\mu \in M$ evaluates to $\omega(m-\ell,\mu) = -\omega(\ell,\mu)$.

	Otherwise said, we have a commutative diagram
	\[
	\xymatrix{
		X^\ast \times X^\ast \ar[r]^-d & k \\
		\Lambda \times M \ar[u]^{\beta_{\Lambda,X} \times  \beta_{M,X}} \ar[r]^-e_{\Lambda,M} & k \ar@{=}[u]
	}
	\]
	where $d: X^\ast \times X^\ast \rightarrow R$ is the adjoint to $d_X(\Lambda,M): X^\ast \rightarrow X$.
\end{proof}

\begin{lemma}\label{lem projbeta}
 Let  $\Lambda_1$ and $\Lambda_2$ be two mutually transverse Lagrangians and $X$ an arbitrary Lagrangian. Denote by $p_1$  (resp. $p_2$ ) the projection map from $X$ onto $\Lambda_1$ (resp. $\Lambda_2$) parallel to $\Lambda_2$ (resp. $\Lambda_1$). Then
	\[
	p_1 = (\beta_{\Lambda_1,\Lambda_2})^{-1}\beta_{X,\Lambda_2} \textrm{ and } p_2 = (\beta_{\Lambda_2,\Lambda_1})^{-1}\beta_{X,\Lambda_1}.
	\]
\end{lemma}
\begin{proof}
	Direct computation.
\end{proof}

\begin{lemma}\label{lem troislag}
	Let $\Lambda_1,\Lambda_2,\Lambda_3$ denote three Lagrangians such that $\Lambda_1 \pitchfork \Lambda_2$ and $\Lambda_2 \pitchfork \Lambda_3$. Let $X$ be an arbitrary  Lagrangien. Then
	\[
	\beta_{X,\Lambda_3}= \beta_{\Lambda_1,\Lambda_3} (\beta_{\Lambda_1,\Lambda_2})^{-1} \beta_{X,\Lambda_2}+\beta_{\Lambda_2,\Lambda_3} (\beta_{\Lambda_2,\Lambda_1})^{-1} \beta_{X,\Lambda_1} 
	\]
\end{lemma}

\begin{proof}
	By direct inspection, with the same notations as in Lemma~\ref{lem projbeta}, 
	\[
	\beta_{X,\Lambda_3} = \beta_{\Lambda_1,\Lambda_3}p_1 + \beta_{\Lambda_2,\Lambda_3}p_2
	\]
	and replacing $p_1$ and $p_2$ by their expressions we get the result.
\end{proof}

We now arrive at the fundamental relation.

\begin{proposition}\label{prop relfond}
	With the same hypothesis as in Lemma~\ref{lem troislag},
	\[
	-(\beta_{\Lambda_2,\Lambda_1})^{-1} \beta_{X,\Lambda_1} + d_{\Lambda_2}(\Lambda_1,\Lambda_3) \beta_{X,\Lambda_2} + (\beta_{\Lambda_2,\Lambda_3})^{-1} \beta_{X,\Lambda_3} = 0.
	\]
\end{proposition}
\begin{proof}
	By composing the relation from Lemma~\ref{lem troislag} by $(\beta_{\Lambda_2,\Lambda_3})^{-1}$ to the left we get:
	\[
	(\beta_{\Lambda_2,\Lambda_3})^{-1} \beta_{X,\Lambda_3} - (\beta_{\Lambda_2,\Lambda_1})^{-1} \beta_{X,\Lambda_1} - (\beta_{\Lambda_2,\Lambda_3})^{-1}\beta_{\Lambda_1,\Lambda_3} (\beta_{\Lambda_1,\Lambda_2})^{-1} \beta_{X,\Lambda_2} = 0
	\]
	And from Lemma~\ref{lem reldiffbeta}
	\[
	- (\beta_{\Lambda_2,\Lambda_3})^{-1}\beta_{\Lambda_1,\Lambda_3} (\beta_{\Lambda_1,\Lambda_2})^{-1}= d_{\Lambda_2}(\Lambda_1,\Lambda_2)
	\]
	whence the result.
\end{proof}

\section{Lagrangian paths and Sturm sequences}\label{sec:chemlagr}

In this section we reformulate some results from Appendix A in~\cite{MR2431916} in terms of Lagrangian paths based at either $L$ or $L^\ast$.
%
%
%
\begin{definition}
	Let $\Lambda$ and $M$  be two Lagrangians. A \emph{Lagrangian path} of length $n$   joining $\Lambda$ to $M$ is a sequence of $n+2$ Lagrangians
	\[
	\Lambda = \Lambda_0,\Lambda_1, \dots, \Lambda_n,\Lambda_{n+1} =M
	\]
	such that for each  $0 \leq i \leq n$ we have $ \Lambda_i \pitchfork \Lambda_{i+1}$.
	
	If $\Lambda = M$ we call it a  \emph{Lagrangian loop}  based at $\Lambda$.

	If $\alpha = \Lambda_0, \dots,\Lambda_n,M$ and $\beta = M, \Lambda'_1, \dots, \Lambda_{m+1}$ are two Lagrangian paths, the first ending at $M$ and the second starting at $M$, then their concatenation is the Lagrangian path  $\alpha \ast \beta : \Lambda_0, \dots,\Lambda_n,M, \Lambda'_1, \dots, \Lambda_{m+1}$. 
\end{definition}

An arbitrary Lagrangian path has an associated Sylvester matrix:
\begin{definition}  \label{def matSylvester}
	Let $\alpha: (\Lambda_0,\Lambda_1, \dots, \Lambda_n,\Lambda_{n+1})$ be a Lagrangian path .  The  Sylvester matrix of  $\alpha$  is the matrix  $S(\alpha)$ of the symmetric bilinear map with support $\Lambda_1^\ast \oplus \cdots \oplus \Lambda_n^\ast$  whose block coefficients are given as follows:
	\[
	\left\{ 
	\begin{array}{c}
	a_{i,i} = d_{\Lambda_i}(\Lambda_{i-1},\Lambda_{i+1}) \textrm{ for } i=1, \dots, n. \\
	a_{i+1,i} = (\beta_{\Lambda_{i+1},\Lambda_i})^{-1}:\Lambda_i^\ast \rightarrow \Lambda_{i+1}  \textrm{ for } i=1, \dots, n-1. \\
	a_{i,i+1}= a_{i+1,i}^\ast = -(\beta_{\Lambda_i,\Lambda_{i+1}})^{-1}
	\end{array}
	\right.
	\]
and all other coefficients are zero.

	Writing $b_{\Lambda_{i+1},\Lambda_{i}} = (\beta_{\Lambda_{i+1},\Lambda_{i}})^{-1}$, the matrix $S(\alpha)$ is a trigonal matrix  of the following form
	
	\[
	\left(
	\begin{matrix}
	d_{\Lambda_1}(\Lambda_0,\Lambda_2) & b_{\Lambda_1,\Lambda_2}& 0 & \cdots & 0 \\
	-b_{\Lambda_2,\Lambda_1} & d_{\Lambda_2}(\Lambda_1,\Lambda_3) & \ddots & \ddots & \vdots  \\
	0 & \ddots &\ddots & \ddots & 0 & \\
	\vdots & \ddots & \ddots  &  \ddots & b_{\Lambda_{n-1},\Lambda_{n}} \\
	0 & \cdots  &0&-b_{\Lambda_{n},\Lambda_{n-1}} & d_{\Lambda_n}(\Lambda_{n-1},\Lambda_{n+1})
	\end{matrix}
	\right) 
	\]
	
\end{definition}

The map that to a Lagrangian path associates its Sylvester matrix has an obvious equivariance property, which we formally state for future reference, the proof is trivial.

\begin{lemma}\label{lem:sylvis equiv}
Let $
\alpha : (\Lambda_0, \Lambda_1, \dots, \Lambda_n,\Lambda_{n+1})$
be a Lagrangian path and $\phi \in Sp_{2g}(k)$, then $
\phi_\ast(\alpha): (\phi (\Lambda_0), \phi(\Lambda_1), \dots, \phi(\Lambda_n),\phi(\Lambda_{n+1}))
$ 
is again a Lagrangian path, and the map $\phi$ induces a canonical ismetry between the associated Sylvester  $S(\alpha)$ and $S(\phi_\ast(\alpha))$.
\end{lemma}
The following is a first result showing how much of the properties of a Lagrangian path is encoded in its Sylvester matrix. 

\begin{proposition}\label{prop:transeqnondeg}
	Let $\alpha: \Lambda_0, \cdots, \Lambda_{n+1}$ be a Lagrangian path and denote by  $S(\alpha)$ its Sylvester matrix. To reduce the amount of notation, we  write $\beta_{i,j}$ for $\beta_{\Lambda_i,\Lambda_j}$  and we consider the morphism
	\[
	E_0: \Lambda_0  \rightarrow  \Lambda_1^\ast \oplus \cdots \oplus \Lambda_{n}^\ast,
	\]
	whose matrix is the transpose of the row
	\[
	\left[
	\begin{matrix}
		\beta_{0,1} &
		\beta_{0,2} &
		\dots &
		\beta_{0,n}
	\end{matrix}
	\right]
	\]
	and
	\[
	T_n: \Lambda_n^\ast \rightarrow \Lambda_1 \oplus \cdots \oplus\Lambda_{n}
	\]
	whose matrix is the transpose of the row
	\[
	\left[
	\begin{matrix}
		0 &
		\dots &
		0 &
		-(\beta_{n,n+1})^{-1}
	\end{matrix}
	\right].
	\]
	Then
	\begin{enumerate}
		\item Both  morphisms $E_0$ and $T_n$ are injective.
		\item The following diagramm, viewed as a chain map between horizontal chain complexes, is a homotopy equivalence.
	\end{enumerate}
	\[
	\xymatrix{
		\Lambda_0 \ar[r]^-{\beta_{0,n+1}} \ar[d]^-{E_0} & \Lambda_{n+1}^\ast \ar[d]^{T_n} \\
		\Lambda_1^\ast \oplus \cdots \oplus \Lambda_{n}^\ast \ar[r]_{S(\alpha)} & \Lambda_1 \oplus \cdots \oplus \Lambda_{n}
	}
	\]
\end{proposition}
\begin{proof} $i)$ Since $\Lambda_0 \pitchfork \Lambda_1$, the morphism $\beta_{0,1}$ is an isomorphism, and as it is the first component of the matrix of  $E_0$ this is an injective morphism. The morphism $T_0$ is also injective: its only non-zero component is an isomorphism because the underlying Lagrangians that determine it are transverse.
	We now check the commutativity of the diagram, and for this we first compute $S(\alpha) E_0$. Again to reduce the amount of notation, we write  $d_{j}(j-1,j+1)$ instead of $d_{\Lambda_j}(\Lambda_{j-1},\Lambda_{j+1})$. 
	
	The first coefficient of the column matrix $S(\alpha)E_0$ is 
	\[
	d_1(0,2) \beta_{0,1}  +\beta_{1,2}^{-1}\beta_{0,2}.
	\]
	As  $\Lambda_0 \pitchfork \Lambda_1$, we know that  $\beta_{0,1}$ is invertible, and Lemma~\ref{lem reldiffbeta} shows that the coefficient is in fact zero.
	
	For $2 \leq i \leq n-1$, the $i$-th coefficient is
	\[
	-\beta_{i,i-1}^{-1} \beta_{0,i-1} + d_{i}(i-1,i+1)\beta_{0,i} + \beta_{i,i+1}^{-1}\beta_{0,i+1} 
	\]
	which is zero according to Proposition~\ref{prop relfond}.
	
	Finally, the last coefficient is equal to
	\[
	-\beta_{n,n-1}\beta_{0,n-1} + d_n(n-1,n+1)\beta_{0,n}
	\]
	which, by Proposition~\ref{prop relfond}, is equal to
	\[
	-\beta_{n,n+1}^{-1}\beta_{0,n+1}.
	\]
	And this is exactly the unique non-zero coefficient of the column matrix $T_n\beta_{0,n+1}$.
	
	$ii)$ Consider now the projection over the first $n-1$ coordinates:
	\[
	p: \Lambda_1 \oplus \cdots \oplus \Lambda_n \rightarrow  \Lambda_1 \oplus \cdots \oplus \Lambda_{n-1} 
	\]
	and
	\[
	q: \Lambda_1^\ast \oplus \cdots \oplus \Lambda_n^\ast  \rightarrow  \Lambda_2^\ast \oplus \cdots \oplus \Lambda_{n}^\ast
	\]
	which is given by the matrix
	\[
	\left[
	\begin{matrix}
		-\beta_{0,2}\beta_{0,1}^{-1} & 1 & 0  & \dots & 0 \\
		-\beta_{0,3}\beta_{0,1}^{-1} & 0 & 1 & \ddots  & \vdots \\
		\vdots & \vdots & \ddots & \ddots &0 \\
		-\beta_{0,n}\beta_{0,1}^{-1} & 0 & \dots  & 0 & 1 \\
	\end{matrix}
	\right]
	\]
	The identity block on the right hand side of the matrix shows that $q$ is surjective and a direct computation shows that $qE_0 = 0$, hence by dimensional reasons  $q = \operatorname{coker} E_0$. Furthermore,  trivially, $p = \operatorname{coker} T_n$.
	
	Let finally $s:  \Lambda_2^\ast \oplus \cdots \oplus \Lambda_n^\ast  \rightarrow  \Lambda_1 \oplus \cdots \oplus \Lambda_{n-1}$ be the matrix obtained by erasing the first row and the first column in $S(\alpha)$. We have then a commutative diagram with exact rows
	\[
	\xymatrix{
		0 \ar[r] & \Lambda_0 \ar[d] \ar[r]^-{E_0} & \Lambda_1^\ast \oplus \cdots \oplus \Lambda_n^\ast \ar[d]^{S(\alpha)} \ar[r]^-q &  \Lambda_2^\ast \oplus \cdots \oplus \Lambda_n^\ast \ar[r] \ar[d]^s & 0  \\
		0 \ar[r] & \Lambda_{n+1}^\ast \ar[r]^-{T_n} &   \Lambda_1 \oplus \cdots \oplus \Lambda_{n} \ar[r]^p &   \Lambda_1 \oplus \cdots \oplus \Lambda_{n-1} \ar[r] & 0
	}
	\]

	The snake lemma shows then that the morphism of chain complexes $(E_0,T_n)$ is a quasi-isomorphism if and only if  $s$ is an isomorphism. Since the complexes we consider are complexes of $k$-vector spaces,
	any quasi-isomorphism is in fact a homotopy equivalence.
	
	We are left with showing that  $s$  is invertible. The identity block  on the righ-hand side of the matrix  $q $ imposes, by direct computation, that $s$  has to be the matrix that we get from  $S(\alpha)$ by erasing its first column and last row
	\[
	s=
	\left(
	\begin{matrix}
		\beta_{1,2}^{-1}& 0 & \cdots& \cdots & 0 \\
		d_{2}(1,3) & \beta_{2,3}^{-1} & 0 & \cdots &\vdots \\
		-\beta_{3,2}^{-1} & d_{3}(2,4) & \beta_{3,4}^{-1} & \ddots &  \vdots\\
		0 &
		\ddots & \ddots&  \ddots & 0 &  \\
		\cdots & 0&-\beta_{{n-1},{n-2}}^{-1}  &  d_{{n-1}}({n-2},{n}) & \beta_{{n-1},{n}}^{-1} 
	\end{matrix}
	\right).
	\]
	
	Let us show now that $pS(\alpha) = sq$. The matrix for $pS(\alpha)$ is the matrix for $S(\alpha)$ in which we have erased the last row. Hence the unique non-trivial part in the equality $pS(\alpha) = sq$ is the equality between the last two columns of these matrices. We will distinguish between three cases: the first coefficient, the second coefficient and finally the coefficient of index $j,1$ for $j>1$. 
	\begin{enumerate}
		\item[i) ] For the first coefficient, we have
		\[
		d_{1}(0,2) =   -\beta_{1,2}^{-1}\beta_{0,2}\beta_{0,1}^{-1}  \textrm{ by Lemma~\ref{lem reldiffbeta}.}
		\]
		\item[ii)] For the second coefficient, we have
		\[
				-d_2(1,3) \beta_{0,3}\beta_{0,1}^{-1} - \beta_{2,3}^{-1}\beta_{0,3}\beta_{0,1}^{-1}  =   -\beta_{2,1}^{-1} \textrm{ by Proposition~\ref{prop relfond}.}
		\]
		\item[iii)] Finally, when $j>1$,
		\[
		\beta_{j,j-1}\beta_{0,j}\beta_{0,1}^{-1} - d_j(j-1,j+1)\beta_{0,j}\beta_{0,1}^{-1}  - \beta_{j,j+1}^{-1}\beta_{0,j+1}\beta_{0,1}^{-1} = 0  \textrm{ by Proposition~\ref{prop relfond}.}
		\]
	\end{enumerate}

	To conclude observe that as the matrix $s$ is lower triangular with invertible elements along the diagonal it is invertible. 
\end{proof}

Folowwing the same proof nut exchanging the roles of $\Lambda_0$ and $\Lambda_n$ one can show:

\begin{proposition}\label{prop:transeqnondegalt}
	Let $\alpha: \Lambda_0, \cdots, \Lambda_{n+1}$ be a Lagrangian path  $S(\alpha)$ its Sylvester matrix. Write   $\beta_{i,j}$ for $\beta_{\Lambda_i,\Lambda_j}$  and consider the maps
	\[
	F_{n+1}: \Lambda_{n+1}  \rightarrow  \Lambda_1^\ast \oplus \cdots \oplus \Lambda_{n}^\ast
	\]
	with matrix the transposed of the row:
	\[
	\left[
	\begin{matrix}
		\beta_{n+1,1} &
		\beta_{n+1,2} &
		\dots &
		\beta_{n+1,n}
	\end{matrix}
	\right]
	\]
	and
	\[
	U_n: \Lambda_n^\ast \rightarrow \Lambda_1 \oplus \cdots \oplus\Lambda_{n}
	\]
	with matrix the transpose of the row:
	\[
	\left[
	\begin{matrix}
		\beta_{1,0}^{-1} & 
		0 &
		\dots &
		0 
	\end{matrix}
	\right]
	\]
	Then 
	\begin{enumerate}
		\item The maps $F_{n+1}$ and $U_n$ are injective.
		\item The following diagram, viewed as chain maps between horizontal chain complexes is a homotopy equivalence.
	\end{enumerate}
	\[
	\xymatrix{
		\Lambda_{n+1} \ar[r]^-{\beta_{n+1,0}} \ar[d]^-{F_{n+1}} & \Lambda_{0}^\ast \ar[d]^{U_n} \\
		\Lambda_1^\ast \oplus \cdots \oplus \Lambda_{n}^\ast \ar[r]_{S(\alpha)} & \Lambda_1 \oplus \cdots \oplus \Lambda_{n}
	}
	\]
\end{proposition}

The following result is the translation in the context of general Lagrangian paths of  Proposition-Definition A.2.1 in~\cite{MR2431916}.
\begin{corollary}\label{cor:sylvdettrans}
Let $\alpha: (\Lambda_0, \Lambda_1,\dots,\Lambda_n,\Lambda_{n+1})$ be a Lagrangian path and let  $S(\alpha)$ denote its Sylvester matrix. The following are equivalent:
	\begin{enumerate}
		\item The bilinear form $S(\alpha)$ is non-degenerate.
		\item The two Lagrangians $\Lambda_0$ and $\Lambda_{n+1}$ are transverse.
	\end{enumerate}
\end{corollary}
\begin{proof} By Lemma~\ref{lem bildeuxlag} (4), the map $\beta_{0,n+1}$ is invertible if and only if $L_0 \pitchfork L_{n+1}$, and in the context of Proposition~\ref{prop:transeqnondeg} this happens if and only if $S(\alpha)$ is invertible too.	
\end{proof}

The next is the key result to simplify Sylvester matrices of long Lagrangian paths.

\begin{lemma}["Shorcut Lemma"]\label{lem:raccourci}
	Let $\Lambda_0, \cdots, \Lambda_{n+1}$ be a Lagrangian path with Sylvester matrix $S$. We suppose that there are two  indices $0 \leq i < j \leq n+1$ such that $\Lambda_i \pitchfork \Lambda_j$.  We then have two more Lagrangian paths: 
	\begin{enumerate}
		\item[a)] The sub-path $\Lambda_i \cdots, \Lambda_j$, with Sylvester matrix $S(\Lambda_i,\dots,\Lambda_j)$ and
		\item[b)] The shortened path $\Lambda_0, \cdots, \Lambda_i,\Lambda_j, \cdots \Lambda_{n+1}$, with sylvester matrix
		
		$S(\Lambda_0,\dots,\Lambda_i,\Lambda_j,\dots,\Lambda_{n+1})$.
	\end{enumerate}
	Then
	\[
	S \textrm{ is isometric to }S(\Lambda_i,\dots,\Lambda_j) \oplus S(\Lambda_0,\dots,\Lambda_i,\Lambda_j,\dots,\Lambda_{n+1}).
	\]
\end{lemma}
\begin{proof}
	
	We have three cases to consider: the case where $i = 0$, where $j=n+1$ and finally $0<i<  j <n+1 $. To simplify our notations, for $0\leq s < t \leq n+1$ we denote by $S(s,t)$ the matrix $S(\Lambda_s, \dots, \Lambda_t )$.

	$\bullet$ First case:  $i=0, j<n+1$.
	
	The support of the matrix $S(0,n+1)$ is  
	\[
	\Lambda_1^\ast \oplus \Lambda_{j-1}^\ast\oplus\Lambda_j^\ast \cdots \oplus \Lambda_n^\ast
	\]
	and its restriction to
	$\Lambda_1^\ast\oplus \cdots  \oplus \Lambda_{j-1}^\ast$, is by construction $S(0,j) $, which is non-degenerate by Corollary~\ref{cor:sylvdettrans}. In particular, 
	\[
	S(0,{n+1}) = S(0,j) \bot S(0,j)^\bot.
	\] 
	Let us compute the orthogonal of $\Lambda_1^\ast \oplus \cdots \oplus \Lambda_{j-1}^\ast$ with respect to $S(0,n+1)$. As this matrix is trigonal, the subspace $\Lambda_{j+1}^\ast \oplus\cdots \oplus  \Lambda_{n}^\ast$ is in the orthogonal. 
	Commutativity of the following diagram
	\[
	\xymatrix{
		\Lambda_j^\ast \ar[dr]^{\beta_{0,j}^{-1}} & & \\
		&\Lambda_0 \ar[r]^-{\beta_{0,j+1}} \ar[d]^-{E_0} & \Lambda_{j+1}^\ast \ar[d]^{[0,\dots,0,-(\beta_{j,j+1})^{-1}] }\\
		&**[l]  \Lambda_1^\ast \oplus \cdots \oplus \Lambda_{j}^\ast \ar[r]_{S(0,j+1)} & **[r]\Lambda_1 \oplus \cdots \oplus \Lambda_{j}
	}
	\]
	shows that the copy of  $\Lambda_j^\ast$ included in the support of $S(0,n+1)$ via the map  $E_{0,j}\beta_{0,j}^{-1}$  is also in the orthogonal of $\Lambda_1^\ast \oplus \Lambda_{j-1}$; indeed its image under the matrix   $S(0,n+1)$ has as first  $j-1$ coefficients equal to zero. Since moreover this copy of $\Lambda_j^\ast$  is in direct sum with the preceding subspace, by dimensional reasons  the orthogonal is $\Lambda_j^\ast \oplus \Lambda_{j+1} \oplus \cdots \oplus \Lambda_{n}^\ast$. Let us now compute the restriction of $S(0,n+1)$ to this subspace.

	Observe that the inclusion of the orthogonal  to $\Lambda_{j}^\ast \oplus \dots \oplus \Lambda_{n}^\ast$ into the support $S(0,n+1)$ is given by the matrix:
	\[
	P= \left(
	\begin{matrix}
		E_0 \beta_{0,j}^{-1} & 0 \\
		0 &  Id_{n-j-1}
	\end{matrix}
	\right),
	\]
	with domain $\Lambda_j^\ast \oplus \cdots \oplus \Lambda_n^\ast$ and codomain $\Lambda_1 \oplus \cdots \oplus \Lambda_{n}$.
	Let us write the matrix $S(0,n+1)$ by blocks.
	\[
	S=
	\left(
	\begin{matrix}
		S(0,j+1) & A_j \\
		A_j^\ast & S(j,n+1)
	\end{matrix}
	\right)
	\]
	where
	\[
	A_j= 
	\left(
	\begin{matrix}
		0 &  \\
		\vdots & \mbox{\Huge0} \\
		0 & \\
		\beta_{j,j+1}^{-1} & 0 \cdots 0
	\end{matrix}
	\right).
	\]
	
	Then the matrix of $S(0,n+1)$ restricted to $\Lambda_j^\ast \oplus \dots \oplus \Lambda_{n}^\ast$ is
	\[
	P^\ast S(0,n+1)P = \left(
	\begin{matrix}
		\beta_{0,j}^{-1 \ast} E_0^\ast S(0,j+1) E_0 \beta_{0,j}^{-1} & \beta_{0,j}^{-1 \ast} E_0^\ast A_j \\
		A_j^\ast  E_0 \beta_{0,j}^{-1} & S(j,n+1)
	\end{matrix}
	\right)
	\]
	
	To compute the top left corner of the matrix, we use the commutative diagram given by Proposition~\ref{prop:transeqnondeg}:
	\[
	\xymatrix{
		\Lambda_j^\ast \ar[dr]^{\beta_{0,j}^{-1}} & & & & \\
		&\Lambda_0 \ar[r]^-{\beta_{0,j+1}} \ar[d]^-{E_0} & \Lambda_{j+1}^\ast \ar[d]^{T_n } & & \\
		&**[l]  \Lambda_1^\ast \oplus \cdots \oplus \Lambda_{j}^\ast \ar[r]_{S(0,j+1)} & \Lambda_1 \oplus \cdots \oplus \Lambda_{j}  \ar[r]^-{E_0^\ast} & \Lambda_0^\ast \ar[r]^{\beta_{0,j}^{-1 \ast}} & \Lambda_j  
	}
	\]
	which shows that
	\begin{eqnarray*}
		\beta_{0,j}^{-1 \ast} E_0^\ast S(0,j+1) E_0 \beta_{0,j}^{-1}  & = &  \beta_{0,j}^{-1 \ast}E_0^\ast T_n\beta_{0,j+1} \beta_{0,j}^{-1} \\
		& = & -\beta_{j,0}^{-1}\beta_{0,j}^\ast (-\beta_{j,j+1}^{-1})\beta_{0,j+1}\beta_{0,j}^{-1} \\
		& = &  -\beta_{j,0}^{-1}-\beta_{j,0} (-\beta_{j,j+1}^{-1})\beta_{0,j+1}\beta_{0,j}^{-1} \\
		& = & (-\beta_{j,j+1}^{-1})\beta_{0,j+1}\beta_{0,j}^{-1} \\
		& = & d_{j}(0,j+1) \textrm{ by Lemma~\ref{lem reldiffbeta}}
	\end{eqnarray*}
	
	In the same way
	\begin{eqnarray*}
		\beta_{0,j}^{-1 \ast} E_0^\ast A_j&  =&  [\beta_{j,0}^{-1}\beta_{1,0}, \dots, \beta_{j,0}^{-1}\beta_{j,0} ] A_j \\
		& = & [ \beta_{j,j+1}^{-1},  0, \dots,0]
	\end{eqnarray*}
	and so,
	\[
	P^\ast S(0,n+1)P = \left(
	\begin{matrix}
		d_{j}(0,j+1) &  \beta_{j,0}^{-1}\beta_{j,0}\beta_{j,j+1}^{-1},  0, \dots,0 \\
		-\beta_{j+1,j} &  \\
		0 & \\
		\vdots & S(j,n+1) \\
		0 & 
	\end{matrix}
	\right) = S(0,j,\dots,n+1)
	\]

	$\bullet$ For the case where  $i<n, j=n+1$ we proceed in a similar way. The same argument as before shows that
	\[
	S(0,{n+1}) = S(i,n+1)^\bot \bot S(i,n+1),
	\]
	that the subspace $\Lambda_1^\ast \oplus \cdots \oplus \Lambda_{i-1}^\ast$ is in $S(i,n+1)^\bot$ and the commutative diagram
	
	\[
	\xymatrix{
		\Lambda_i^\ast \ar[dr]^{\beta_{n+1,i}^{-1}} & & \\
		&\Lambda_{n+1} \ar[r]^-{\beta_{n+1,i-1}} \ar[d]^-{F_{n+1}} &  \Lambda_{i-1}^\ast \ar[d]^{[\beta_{i,i-1}^{-1}, 0,\dots,0] }\\
		&  **[l]\Lambda_{i}^\ast \oplus \cdots \oplus \Lambda_{n}^\ast \ar[r]_{S(i-1,n+1)} & **[r] \Lambda_{i} \oplus \cdots \oplus \Lambda_{n}
	}
	\]
	shows that the copy of  $\Lambda_i$ included via $ F_{n+1}\beta_{n+1,i}^{-1}$ is in the orthogonal too. 
	
	Then, if we denote by 
	\[
	Q = \left( 
	\begin{matrix}
		Id_{i} & 0 \\
		0 & F_{n+1}\beta_{n+1,i}^{-1}
	\end{matrix}
	\right)
	\]
	the inclusion of $\Lambda_{1}^\ast \oplus \cdots \oplus  \Lambda_{i-1}^\ast \oplus \Lambda_i^\ast$ as the orthogonal to $S(i,n+1)$, we have:
	
	\[
	Q^\ast S(0,n+1)Q =  \left(
	\begin{matrix}
		S(0,i) & A_{i-1}F_{n+1}\beta_{n+1,i}^{-1}  \\  
		\beta_{i,n+1}^{-1}F_{n+1}^\ast A_{i-1}^\ast &  \beta_{n+1,i}^{-1 \ast} F_{n+1}^\ast S(i-1,n+1) F_{n+1} \beta_{n+1,i}^{-1}  
	\end{matrix}
	\right)
	\]
	And, as before,
	\[
	A_{i-1}F_{n+1}\beta_{n+1,i}^{-1} = \left( \begin{matrix} 0 \\ \vdots \\ 0 \\ \beta_{i-1,i}\end{matrix} \right)
	\]
	and using the above commutative diagram:
	\begin{eqnarray*}
		Q^\ast S(i-1,n+1)Q  &= & \beta_{n+1,i}^{-1 \ast} F_{n+1}^\ast S(i-1,n+1) F_{n+1} \beta_{n+1,i}^{-1} \\
		& = & -\beta_{i,n+1}^{-1}\beta_{i,n+1} \beta_{i,i-1}^{-1} \beta_{n+1,i-1} \beta_{n+1,i}^{-1} \\
		& = & \beta_{i,i-1}^{-1} \beta_{n+1,i-1} \beta_{n+1,i}^{-1} \\
		& = & -d_{i}(n+1,i-1) \\
		& = & d_i(i-1,n+1)
	\end{eqnarray*}
	and hence
	\[
	Q^\ast S(0,n+1)Q =  \left(
	\begin{matrix}
		&    0 \\ 
		S(0,i) & \vdots \\
		& 0 \\
		&   \beta_{i-1,i}^{-1}\\  
		0\cdots0-\beta_{i,i-1}^{-1}    & d_{i}(i-1,n+1)
	\end{matrix}
	\right)= S(0\cdots i n+1).
	\]
	
	$\bullet$ We are left the case where $0<i<j<n+1$. Observe that if  $j=i+1$ we have nothing to prove, and so assume that $i+1<j$ and as before we start by computing $S(i,j)^\bot$.

	Here the orthogonal is given by two subspaces. First, as before, the Lagrangians with index too far away from  $i$ and $j$ are in the orthogonal: 
	 \[
	 (\Lambda_1^\ast \oplus \cdots \oplus \Lambda_{i-1}^\ast) \bigoplus (\Lambda_{j+1}^\ast \oplus \cdots \oplus \Lambda_n^\ast) \subset S(i,j)^\bot.
	 \]
	
	The same computation as in the first case shows also that the subspace $\Lambda_i^\ast$ included via the composite
	\[
	\xymatrix{
		\Lambda_i^\ast \ar[r]^{\beta_{j,i}^{-1}} & \Lambda_j \ar[r]^-{F_j} & \Lambda_i^\ast \oplus \cdots \oplus \Lambda_{j-1}^\ast
	}
	\]
	is in the orthogonal; and as in the second case, the subspaces $\Lambda_j^\ast$ included via:
	\[
	\xymatrix{
		\Lambda_j^\ast \ar[r]^{\beta_{i,j}^{-1}} & \Lambda_i \ar[r]^-{E_i} & \Lambda_{i+1}^\ast \oplus \cdots \oplus \Lambda_{j}^\ast
	}
	\]
	are in the orthogonal too. This orthogonal $S(i,j)^\bot$ is therefore isomorphic to $(\Lambda_1^\ast \oplus \cdots \oplus \Lambda_{i}^\ast) \bigoplus (\Lambda_{j}^\ast \oplus \cdots \oplus \Lambda_n^\ast)$, and the matrix of the inclusion of this subspace into the support of $S(0,n+1)$ is given, writing blocks according to the decomposition of the source into $(\Lambda_1^\ast \oplus \cdots \oplus \Lambda_{i-1}^\ast)\oplus (\Lambda_{i}^\ast \oplus \Lambda_{j}^\ast) \oplus (\Lambda_{j+1}^\ast \oplus \cdots \oplus \Lambda_n^\ast)$ by
	\[
	J=\left(
	\begin{matrix}
		Id & 0   & 0 \\
		0 & E & 0 \\
		0 & 0 & Id 
	\end{matrix}
	\right)
	\]
	where $E$, written by blocks of size   $2 \times (j-i)$, with domain  $\Lambda_i^\ast\oplus \Lambda_j^\ast$ and codomains $\Lambda_{i}^\ast \oplus \cdots \Lambda_{j}^\ast$ is
	\[
	E=  \left(
	\begin{array}{c|c}
		& 0   \\
		F_j\beta_{j,i}^{-1} &  \\ 
		& E_i\beta_{i,j}^{-1} \\
		0 &     
	\end{array}
	\right).
	\]
	
	To compute the restriction of $S(0,n+1)$ to the $S(i,j)^\bot$, we write the matrix $S(0,n+1)$ by blocks  according to the decomposition of the source as $(\Lambda_1^\ast \oplus \cdots \oplus \Lambda_{i-1}^\ast)\oplus (\Lambda_{i}^\ast \oplus \cdots \oplus \Lambda_{j}^\ast) \oplus (\Lambda_{j+1}^\ast \oplus \cdots \oplus \Lambda_n^\ast)$
	\[
	S(0,n+1) = 
	\left(
	\begin{array}{cccc}
		S(0,i) & A_{i-1}   & 0 \\
		A_{i-1}^\ast & S(i-1,j+1) & A_{j} \\
		0 & A_{j}^\ast & S(j,n+1) 
	\end{array}
	\right).
	\]
	The product  $J^\ast S(0,n+1)J $ is:
	
	\[
	\left(
	\begin{array}{cccc}
		S(0,i) & A_{i-1}E  & 0 \\
		E^\ast A_{i-1}^\ast & E^\ast S(i-1,j+1)E & E^\ast A_{j} \\
		0 & A_{j}^\ast E & S(j,n+1) 
	\end{array}
	\right).
	\]
	
	A first direct computation shows that
	\[
	A_{i-1}E =\left( \begin{matrix}  0 & 0 \\ -\beta_{i-1,i}^{-1} & 0\end{matrix} \right) \textrm{ et } A_{j}^\ast E =\left( \begin{matrix}  0 &  \beta_{j+1,j}^{-1}  \\ 0& 0  \end{matrix} \right).
	\]
	
	The diagrams in Propositions~\ref{prop:transeqnondeg} and~\ref{prop:transeqnondegalt}, show that
	
	\begin{align*}
	S(i-1,j+1)E & = 
	\left(
	\begin{matrix}
		\beta_{i+1,i}^{-1}\beta_{j,i}\beta_{j,i}^{-1} & -\beta_{i,i+1}^{-1}\beta_{i,i+1}\beta_{i,j}^{-1} \\
		0  &   0  \\
		\vdots & \vdots \\
		0 & 0 \\
		\beta_{j,j-1}^{-1} \beta_{j,j-1}\beta_{j,i}^{-1} &  \beta_{j,j+1}^{-1}\beta_{i,j+1}\beta_{i,j}^{-1} 
	\end{matrix}
	\right) \\
	& =
	\left(
	\begin{matrix}
		\beta_{i+1,i}^{-1} & -\beta_{i,j}^{-1}  \\
		0  &   0  \\
		\vdots & \vdots \\
		0 & 0 \\
		\beta_{j,i}^{-1} & \beta_{j,j+1}^{-1}\beta_{i,j+1}\beta_{i,j}^{-1} 
	\end{matrix}
	\right).
	\end{align*}
	and analogous computations as those carried out in the first two parts show that 
	\begin{align*}
		E^\ast S(i-1,j+1)E & =
		\left(
		\begin{matrix}
			-\beta_{i,j}^{-1}F_j^\ast S(i-1,j)F_j\beta_{j,i}^{-1}  &   -\beta_{i,j}^{-1}\beta_{i,j}\beta_{i,j}^{-1} \\
			-\beta_{j,i}^{-1}-\beta_{j,i} \beta_{j,j-1}^{-1} \beta_{j,j-1}\beta_{j,i}^{-1} & -\beta_{j,i}^{-1}E_i^\ast S(i,j+1)E_i\beta_{i,j}^{-1}
		\end{matrix}
		\right) \\
		& =   \left(
		\begin{matrix}
			d_i(i-1,j)  &   -\beta_{i,j}^{-1} \\
			\beta_{j,i}^{-1} &  d_j(i,j+1)
		\end{matrix}
		\right) \\
	\end{align*}
	
	To sum up the matrix $J^\ast S(0,n+1) J$ is equal to
	
	\[ 
	\left(
	\begin{array}{ccc|cc|ccc}
		& & & & & & &  \\
		&\Large{ S(0,i)}  & & & && 0 & \\
		&   & &  -\beta_ {i-1,j}^ {-1} & 0 &  &  \\
		\hline
		0 & 0&  \beta_{j,i-1}^{-1} & d_i(i-1,j) & -\beta_{i,j}^{-1} & 0 & 0 & 0\\
		0& 0 & 0 & \beta_{j,i}^{-1}& d_j(i,j+1) & -\beta_{j,j+1}^{-1} & 0 & 0\\
		\hline
		&  &  & 0 & \beta_{j+1,j}^{-1} & &  \\
		& {\Large{ 0}}  & & 0 & 0 & & S(j+1,n+1) & \\
		&  &  &  &   &  & 
	\end{array} 
	\right)
	\]
	and we recognize $S(0\cdots i j \cdots n+1)$.
\end{proof}

The following is an example of the flexibility in manipulating Sylvester matrices that the shortcut lemma provides us with:

\begin{corollary}\label{cor:les4chemins}
Let $\Lambda_0, \Lambda_1, \dots, \Lambda_n, \Lambda_0$ be a Lagrangian loop and let $M$ denote a Lagrangian that is transverse to $\Lambda_0$.  Then in  $W(k)$ the following four Sylvester matrices are equal:
	\[
	\begin{array}{ll}
	(i)\ S(\Lambda_0,\Lambda_1, \dots,\Lambda_n) & (ii)\ S(\Lambda_1,\dots,\Lambda_n, \Lambda_0) \\
	(iii)\ S(\Lambda_0,\dots,\Lambda_n,\Lambda_0,M) & (iv)\  S(M,\Lambda_0,\Lambda_1,\dots,\Lambda_n,\Lambda_0)
	\end{array}
	\]
\end{corollary}
\begin{proof}
	We apply the Shortcut Lemma~\ref{lem:raccourci} to the Lagrangians $\Lambda_0$ and $\Lambda_n$  in the sequence $\Lambda_0,\Lambda_1, \dots \Lambda_n,\Lambda_0,M$; this gives us an isometry:
	\[
	S(\Lambda_0,\Lambda_1,\dots,\Lambda_n,\Lambda_0,M) = S(\Lambda_0,\Lambda_1,\dots,\Lambda_n) \oplus S(\Lambda_0,\Lambda_n,\Lambda_0,M)
	\]
	
	The matrix $S(\Lambda_0,\Lambda_n,\Lambda_0,M)$ has as support the space  $\Lambda_n^\ast \Lambda_0^\ast$ and has the form:
	\[
	\left(
	\begin{matrix}
	d(\Lambda_0,\Lambda_0) & \ \beta \\
	\beta & d(M,\Lambda_0)
	\end{matrix}
	\right) = \left(
	\begin{matrix}
	0 & \ \beta \\
	\beta & d(M,\Lambda_0)
	\end{matrix}
	\right).
	\]
	
	Since $\Lambda_0 \pitchfork N$, $\beta$ is invertible, and this matrix represents a neutral form; this proves that  $(i)$ and $(iii)$ are equal. The equality between $(iii)$ and $(ii)$ can be shown similarly using that  $\Lambda_1$ and the second copy of  $\Lambda_0$ in the sequence  \[
	\Lambda_0,\Lambda_1, \dots \Lambda_n,\Lambda_0,M\]
	are transverse. Finally, the same argument but using that  $
	\Lambda_0 \pitchfork \Lambda_n$ in the sequence $M,\Lambda_0, \dots,\Lambda_n,\Lambda_0$ shows that $(iv)$ and $(i)$ are equal.
\end{proof}

\section{The Maslov index}\label{sec:Maslovindex}

\subsection{Index of a Lagrangian path}
For a Lagrangian path the definition of the Maslov index is straightforward.
\begin{definition}\label{def indicechemin}
	Let $\alpha: \Lambda_0, \Lambda_1, \dots, \Lambda_n,\Lambda_{n+1}$ be a Lagrangian path. The class in $W(k)$ of the Sylvester matrix  $S(\alpha)$ is the Maslov index of  $\alpha$, we denote it by  $\Mas(\alpha)$, or $\Mas(\Lambda_0, \cdots,\Lambda_{n+1})$ if needed.
\end{definition}

In contrast, for a Lagrangian \emph{loop}, the Shortcut Lemma~\ref{lem:raccourci}, gives 4 different ways to compute its  Maslov index.

\begin{lemma}\label{lem:diffpathloopirrelevant}
	Let $\omega:\Lambda_0, \Lambda_1, \dots, \Lambda_n,\Lambda_{0}$ be a Lagrqangian loop bases at $\Lambda_0$, and let $M \pitchfork\Lambda$ be an arbitrary Lagrangian.  Then in $W(k)$ the four Sylvester matrices
		\[
	\begin{array}{ll}
		(i)\ S(\Lambda_0,\Lambda_1, \dots,\Lambda_n) & (ii)\ S(\Lambda_1,\dots,\Lambda_n, \Lambda_0) \\
		(iii)\ S(\Lambda_0,\dots,\Lambda_n,\Lambda_0,M) & (iv)\  S(M,\Lambda_0,\Lambda_1,\dots,\Lambda_n,\Lambda_0)
	\end{array}
	\]
	define belong to the class $\Mas(\omega)$.
\end{lemma}
\begin{proof}
	We only write down the proof for $i)$, the proof for the other $3$ cases are similar. By the shortcut Lemma~\ref{lem:raccourci}, as bilinear forms:
	\begin{eqnarray*}
		S(\Lambda_0,\cdots,\Lambda_0) & = & S(\Lambda_0,\cdots,\Lambda_{n}) \bot S(\Lambda_0,\Lambda_n,\Lambda_0).
	\end{eqnarray*}
Now $(\Lambda_0,\Lambda_n,\Lambda_0)$ is the trivial bilinear form with support $\Lambda_n^\ast$, and hence in $W(k)$ (i.e. after regularization), $\Mas(\omega)= \Mas(\Lambda_0,\cdots,\Lambda_0) = \Mas(\Lambda_0,\cdots,\Lambda_{n})$.
\end{proof}

We know explicit the relation between concatenation of paths and the Maslov index.

\begin{proposition}\label{prop actilacets}
	Let $\Lambda$ be a Lagrangian, let $\alpha$ be a Lagrangian path endining at  $\Lambda$, $\beta$ a Lagrangian path starting at $\Lambda$ and $\omega$ a Lagrangian loop pased at $\Lambda$ (here $\alpha$ and $\beta$  could be loops or epmty). Then 
	\[
	S(\alpha \ast \omega \ast \beta) = S(\omega) \bot S(\alpha \ast \beta),
	\]
	and hence in $W(k)$
	\[
	\Mas(\alpha \ast \omega \ast \beta) = \Mas(\omega) + \Mas(\alpha \ast \beta).
	\]
\end{proposition}
\begin{proof}
	If $\alpha = \Lambda_0,\dots,\Lambda_{n},\Lambda$, $\omega = \Lambda,\Lambda_{n+2}\dots, \Lambda_{n+s},\Lambda$ and $\beta = \Lambda,\Lambda'_{n+2},\dots,\Lambda'_{n+p}$, then in the concatenation
	\[ \alpha \ast\omega \ast  \beta :\Lambda_0,\dots,\Lambda_{n},\Lambda ,\Lambda_{n+2}\dots,\Lambda_{n+s},\Lambda,\Lambda'_{n+2},\dots,\Lambda'_{n+p}
	\] we observe that by construction $\Lambda_{n+2} \pitchfork \Lambda$, the starting Lagrangian for $\beta$. Hence, by the Shortcut Lemma~\ref{lem:raccourci},
	\begin{equation*}
		\begin{split}
			S(\alpha \ast \omega \ast \beta) & = S(\Lambda_{n+2},\dots,\Lambda_{n+1})\bot  S(\Lambda_0,\dots,\Lambda_{n+1},\Lambda'_{n+2}\dots,\Lambda'_{n+1},\Lambda'_{n+p}) \\
		&	= S(\omega) \bot S(\alpha \ast \beta).
		\end{split}
	\end{equation*}
\end{proof}

\begin{proposition}\label{prop:relsdansMWR}
	Let $\alpha:  \Lambda_0, \dots,\Lambda_{n+1}$ be a Lagrangian path, let  $\alpha^{-1}$ be its the inverse path and $\beta: \Lambda_0,M_1, \dots, M_m,\Lambda_0$ a loop based at  $\Lambda_0$. Then in $W(k)$
	\begin{enumerate}
		\item $S(\alpha \ast \alpha^{-1}) = 0$.
		\item $S(\alpha) = S(\beta) \bot S(\beta^{-1}\ast \alpha)$.
		\item $S(\alpha) = -S(\alpha^{-1})$.
	\end{enumerate}
\end{proposition}
\begin{proof}
	\begin{enumerate}
		\item  We argue by induction on the length  $n$ of the path $\alpha$.
		
		If $n=0$, we have two transverse Lagrangians $L  \pitchfork M$, and $\alpha \ast \alpha^{-1} = LML$, which has a zero Sylvester matrix.
		
		If $n=1$, $\alpha \ast \alpha^{-1} = L,L_1,M,L_1,L$ , its Sylvester matrix $\alpha\ast\alpha^{-1}$ is a matrix with domain $(L_1,M)$ and codomain $(L_1^\ast,M^\ast)$:
		\[
		\left(
		\begin{matrix}
		d_{M,L} & b_{L,M}  \\
		b_{M,L}  & -d_{M,L}
		\end{matrix}
		\right)
		\]
		As $L \pitchfork L_1$, this is a non-degenerate form. An immediate computation shows then that both $L_1$ and $M$ are isotropic for this form which is therefore neutral.
		
		For the general case, let $n\geq 1$ and choose $\alpha$, a path of length $n+1$. In the middle of the loop  $\alpha \ast \alpha^{-1} = L_0,L_1, \dots,L_n L_{n+1}.M,L_{n+1},L_n \dots, L_0$, we find a sub-path  $\beta = L_n,L_{n+1},M,L_{n+1}$ that joins two mutually transverse Lagrangians. Let $\gamma = L_0,L_1,\dots, L_n,L_{n+1}$, it is a path of length  $n$, and observe that $ L_0,L_1, \dots,L_{n},L_{n+1},L_{n} \dots, L_0 = \gamma \ast  \gamma^{-1}$. By the Shortcut Lemma~\ref{lem:raccourci}
		\[
		S(\alpha \ast \alpha^{-1}) = S(\gamma \ast \gamma^{-1}) \bot S(\beta \ast \beta^{-1})
		\]
		where by induction the two rightmost terms are zero in $W(k)$.
		\item The path $\beta \ast \beta^{-1}\ast \alpha$ is given by the sequence
		\[
		\Lambda_0,M_1 \dots M_m ,\Lambda_{0},M_m \dots, M_1,\Lambda_0, \Lambda_1,\dots, \Lambda_n,\Lambda_{n+1}
		\]
		As  $\Lambda_0 \pitchfork \Lambda_1$ the shortcut Lemma~\ref{lem:raccourci} allows to write
		\begin{eqnarray*}
			S(\beta \ast \beta^{-1} \ast \alpha) & = & S( \Lambda_0,M_1 \dots,\Lambda_0, ,\Lambda_1) \bot S(\Lambda_{0},\Lambda_1 \dots, \Lambda_n,\Lambda_{n+1})\\
			& = & S(\beta \ast \beta^{-1}) \bot S(\alpha) \textrm{ by Proposition~\ref{prop actilacets}} \\
			& = &  S(\alpha) \textrm{ by part } 1. \\
		\end{eqnarray*}
		The, using that $\Lambda_0 \pitchfork M_m$ and Proposition~\ref{prop actilacets}, we have that
		\begin{eqnarray*}
			S(\beta \ast \beta^{-1} \ast \alpha) & = & S( \Lambda_0,M_1 \dots,\Lambda_0,M_m) \bot S(\Lambda_0,M_m,\dots,M_1,\Lambda_0,\Lambda_1, \dots, \Lambda_{n+1}) \\
			& = & S(\beta) \bot S(\beta^{-1}\ast \alpha)
		\end{eqnarray*}
		
		\item Direct from the definitions and Lemma~\ref{lem bildeuxlag}.
	\end{enumerate}

\end{proof}

\subsection{Index  of a triple of Lagrangians}~\label{subsec:indtriple}
%

\begin{proposition}\label{prop defindice}
	Fix three Lagrangians $\Lambda_0,\Lambda_1,\Lambda_2$. Choose two Lagrangian paths $\alpha_{01}$ and $ \alpha_{12}$, where $\alpha_{ij}$ joins $\Lambda_i$ to $\Lambda_j$. Then the class in $W(k)$ of the bilinear symmetric form 
	\[
	\mu_{BL}(\Lambda_0,\Lambda_1,\Lambda_2) =  S(\alpha_{01} \ast \alpha_{12}) \bot -S(\alpha_{01}) \bot -S(\alpha_{12})
	\]
	is independent of the choice of paths. It is by definition the \emph{Maslov index} of the three Lagrangians.
\end{proposition}
\begin{proof} We only treat the case where we change the path  $\alpha_{01}$ by an alternate path  $\alpha'_{01}$, the other case is analogous, and from these two the general case follows immediately. 
	
	By construction we have a loop  $\alpha_{01} \ast {\alpha'}_{01}^{-1}$ based at $\Lambda_0$, and by Proposition~\ref{prop:relsdansMWR} applied to this loop and to the paths $\alpha_{01} \ast \alpha_{12}$ and $\alpha_{01}$ we get
	\begin{eqnarray*}
		\mu_{BL}(\Lambda_0,\Lambda_1,\Lambda_2) & = &
		S(\alpha_{01} \ast \alpha_{12}) \bot -S(\alpha_{01}) \bot -S(\alpha_{12}) \\
		& =  & S(\alpha'_{01} \ast \alpha_{01}^{-1}) \bot S(\alpha'_{01} \ast \alpha_{01}^{-1} \ast \alpha_{01} \ast \alpha_{12}) \bot -S(\alpha'_{01} \ast \alpha_{01}^{-1}) \\
		& &  \bot -S(\alpha'_{01} \ast \alpha_{01}^{-1} \ast \alpha_{01}) \bot -S(\alpha_{12}) \\
		& = & S(\alpha'_{01} \ast \alpha_{01}^{-1})\bot -S(\alpha'_{01} \ast \alpha_{01}^{-1}) \bot S(\alpha'_{01}   \ast \alpha_{12})\bot S(\alpha_{01} \ast \alpha_{01}^{-1})  \\ 
		& &\bot -S(\alpha_{01} \ast \alpha_{01}^{-1}) \bot -S(\alpha'_{01} ) \bot -S(\alpha_{12}) \textrm{ (Prop.~\ref{prop actilacets})}\\
		& = & S(\alpha'_{01}  \ast \alpha_{12}) \bot -S(\alpha'_{01} ) \bot -S(\alpha_{12})
	\end{eqnarray*}
where the last equality comes from the fact that if $\gamma$ is a loop, the $S(\gamma)$is invertible, and hence $S(\gamma)\bot-S(\gamma)$ is neutral. 
\end{proof}

We now check that $\mu_{BL}$ satisfies the characteristic properties of the Maslov index~\cite{CLM}, point $5)$ settles the multiplicative ambiguity.  

\begin{theorem}\label{thm propMaslov}
	The Maslov index  $\mu$ of three Lagrangians satisfies the following properties:
	\begin{enumerate}
		\item If two of the three Lagrangians coincide then  $\mu_{BL}(\Lambda_0,\Lambda_1,\Lambda_2) = 0$.
		\item For any $g \in Sp_{2g}(k)$,  $\mu_{BL}(g\Lambda_0,g\Lambda_1,g\Lambda_2) = \mu_{BL}(\Lambda_0,\Lambda_1,\Lambda_2).$
		\item The Maslov index is a  $2$-cocycle; if $\Lambda_0,\Lambda_1,\Lambda_2,\Lambda_3$ are $4$ Lagrangiens then
		\[
		\mu_{BL}(\Lambda_1,\Lambda_2,\Lambda_3) - \mu_{BL}(\Lambda_0,\Lambda_2,\Lambda_3) + \mu_{BL}(\Lambda_0,\Lambda_1,\Lambda_3) -\mu_{BL}(\Lambda_0,\Lambda_1,\Lambda_2) = 0
		\]
		\item If $\sigma \in \mathfrak{S}_3$ is a  permutation of the indices $0,1$ and $2$ with signature $\varepsilon(\sigma)$, then
		\[
		\mu_{BL}(\Lambda_{\sigma(0)},\Lambda_{\sigma(1)},\Lambda_{\sigma(2)})= \varepsilon(\sigma)\mu_{BL}(\Lambda_0,\Lambda_1,\Lambda_2).
		\]
		\item Let $\mu_{LK}$ denote the regularization of the Wall-Kashiwara index of $3$ Lagrangians, then $2\mu_{BL}=\mu_{LK}.$
	\end{enumerate}
	
\end{theorem}
\begin{proof}
	\begin{enumerate}
		\item If $\Lambda_0 = \Lambda_1$, Then choose a Lagrangian $M$ transverse to $\Lambda_0$, this defines a path  $\alpha_{01} = \Lambda_0 M \Lambda_0$, choose a an arbitrary path $\alpha_{12}$. As the Sylvester matrix of the loop $\alpha_{01}$ is zero, by the Shortcut Lemma~\ref{lem:raccourci} and after regularization $\alpha_{01} \ast \alpha_{12} = \alpha_{12}$ and therefore
		\begin{eqnarray*}
			\mu_{BL}(\Lambda_0,\Lambda_1,\Lambda_2)  &  = & S(\alpha_{01} \ast \alpha_{12}) \bot -S(\alpha_{01}) \bot -S(\alpha_{12}) \\
			& = & S( \alpha_{12})  \bot -S(\alpha_{12}) \\
			& = & 0.
		\end{eqnarray*}
		when $\Lambda_1 = \Lambda_2$ the proof is as before. And if  $\Lambda_0 = \Lambda_2$, choose an arbitrary path $\alpha_{01}$ and choose as path $\alpha_{12}$ the path $\alpha_{01}^{-1}$. Proposition~\ref{prop actilacets}  implies then that the index computed with these paths is zero.
		\item This is an immediate consequence of the the equivariance properties of the Sylvester matrices, Lemma~\ref{lem:sylvis equiv}.
		\item We choose  $3$ paths $\alpha_{01}$, $\alpha_{12}$ and $\alpha_{23}$ and we denote by $D(\Lambda_0, \Lambda_1,\Lambda_2,\Lambda_3)$ the alternating sum.  then, 
		\begin{eqnarray*}
			D(\Lambda_0, \Lambda_1,\Lambda_2,\Lambda_3) & =&  S(\alpha_{12} \ast \alpha_{23}) \bot -S(\alpha_{12}) \bot -S(\alpha_{23})  \\
			& &  \bot-S(\alpha_{01} \ast \alpha_{12} \ast \alpha_{23}) \bot S(\alpha_{01} \ast \alpha_{12}) \bot S(\alpha_{23})\\
			& & \bot S(\alpha_{01} \ast \alpha_{12}\ast \alpha_{23}) \bot -S(\alpha_{01}) \bot -S(\alpha_{12}\ast \alpha_{23}) \\
			&&-S(\alpha_{01} \ast \alpha_{12}) \bot S(\alpha_{01}) \bot -S(\alpha_{12}) \\
		\end{eqnarray*}
	In this big orthogonal sum each bilinear form appears twice with opposite signs, and hence after regularization it is trivial.
		\item It is enough to show the statement for a transposition.  Applying the cocycle result to the four Lagrangiens $\Lambda_0,\Lambda_1,\Lambda_0,\Lambda_2$ we get
		\begin{align*}
			0 & =  \mu_{BL}(\Lambda_1,\Lambda_0,\Lambda_2) - \mu_{BL}(\Lambda_0,\Lambda_0,\Lambda_2) + \mu_{BL}(\Lambda_0,\Lambda_1,\Lambda_2) -\mu_{BL}(\Lambda_0,\Lambda_1,\Lambda_0) \\
			& =  \mu_{BL}(\Lambda_1,\Lambda_0,\Lambda_2) + \mu_{BL}(\Lambda_0,\Lambda_1,\Lambda_2).
		\end{align*}
		The same computation applied to  $\Lambda_0,\Lambda_1,\Lambda_2,\Lambda_0$ and $\Lambda_0,\Lambda_1,\Lambda_2,\Lambda_1$ gives the result for the other two transpositions.
		\item Since $k$ is a field with at least $3$ elements by~\cite[Lemma 1.6 ]{MR1740881}there exists a Lagrangien $\Lambda'$ which is simultaneously transverse to the  three Lagrangians $\Lambda_0,\Lambda_1,\Lambda_2$.
		
		We choose the following paths: $\alpha_{01} : \Lambda_0, \Lambda' \Lambda_1$,  and $\alpha_{12}:  \Lambda_1, \Lambda',\Lambda_2$. By definition, $\alpha_{01} \ast \alpha_{12} =  \Lambda_0, \Lambda',\Lambda_1, \Lambda',\Lambda_2$ and 
		\[
		\mu_{BL}(\Lambda_0,\Lambda_1,\Lambda_2) = S( \Lambda_0, \Lambda',\Lambda_1, \Lambda',\Lambda_2) \bot -S( \Lambda_0, \Lambda',\Lambda_1) \bot -S(\Lambda_1, \Lambda',\Lambda_2)
		\]
		By the Shortcut Lemma~\ref{lem:raccourci}, 
		\[
		S( \Lambda_0, \Lambda',\Lambda_1, \Lambda',\Lambda_2) = S( \Lambda_0, \Lambda',\Lambda_1, \Lambda',) \bot S( \Lambda_0, \Lambda',\Lambda_2)
		\]
		More over, as $\Lambda_0\pitchfork \Lambda'$, the bilinear form $S( \Lambda_0, \Lambda',\Lambda_1, \Lambda') $ is non-degenerate. It is a form with  support $\Lambda'\oplus\Lambda_1$, and with matrix:
		\[
		\left(
		\begin{matrix}
		d_{\Lambda_1,\Lambda_0} & \beta_{\Lambda_1,\Lambda'}  \\
		\beta_{\Lambda',\Lambda_1}  & d_{\Lambda',\Lambda'}
		\end{matrix}
		\right) =
		\left(
		\begin{matrix}
		d_{\Lambda_1,\Lambda_0} & \beta_{\Lambda_1,\Lambda'}  \\
		\beta_{\Lambda',\Lambda_1}  & 0
		\end{matrix}
		\right),
		\]
		As $\beta_{\Lambda_1,\Lambda'}$ is invertible, this is the matrix of a neural form.
		
		Finally:
		\[
		\mu_{BL}(\Lambda_0,\Lambda_1,\Lambda_2) = S( \Lambda_0, \Lambda',\Lambda_2) \bot -S( \Lambda_0, \Lambda',\Lambda_1) \bot -S(\Lambda_1, \Lambda',\Lambda_2)
		\]
	is the regularization of a bilinear form with support  $\Lambda'\oplus\Lambda'\oplus\Lambda'$ and with matrix
		\[
		\left(
		\begin{matrix}
		d_{\Lambda_2,\Lambda_0} & 0 & 0  \\
		0 & d_{\Lambda_1,\Lambda_0}  & 0 \\
		0 & 0& d_{\Lambda_1,\Lambda_2}
		\end{matrix}
		\right)
		\]
		For an explicit computation that shows that this is indeed half Kashiwara-Wall's index we refer the reader to~\cite[Proposition 7.8.3]{MR3012162}
	\end{enumerate}
\end{proof}

\section{Triviality mod $I^2$ of Maslov's cocycle}\label{sec:trivmodI2}

Consider our prefered Lagrangian $L$, although any other would equally suit our purposes. By point $(3)$ in Theorem~\ref{thm propMaslov}, the map $Sp_{2g}(k)\times Sp_{2g}(k) \rightarrow W(k)$ that to a pair $(A,B)$ associates $\mu_{BL}(L,AL,ABL)$ is a $2$-cocycle, and hence determines a group extension
\[
\xymatrix{
0 \ar[r] & W(k) \ar[r] \Gamma \ar[r] & Sp_{2g}(k) \ar[r] & 1.
}
\]
In this section we identify this extension of this class as the push-out of a canonical extension defined from the stabilizers of the Lagrangians $L$ and $L^\ast$, we show that the reduction mod $I^2$ of the extension splits, and compute explicitly a splitting. Such an explicit splitting is known for $k = \mathbb{R}$ by early work from Turaev~\cite{MR739088}.

\subsection{Sturm sequences and Sylvester matrices}\label{subsec:Sturmand Sylvester}

In Section~\ref{sec:rappels} we introduced the stabilizers $\mathcal{S}_L$ and $\mathcal{S}_{L^\ast}$. To distinguish the identity matrix in both subgroups we will write
\[
\underline{0}  = \left(
\begin{matrix}
	1 & 0 \\
	0& 1
\end{matrix}
\right) \in \mathcal{S}_{L} \textrm{ and }
\overline{0} =  \left(
\begin{matrix}
	1 & 0 \\
	0 & 1
\end{matrix}
\right) \in \mathcal{S}_{L^\ast}.
\]

The free product  $\mathcal{S}_L \ast \mathcal{S}_{L^\ast}$ comes together with an obvious evaluation morphism $E: \mathcal{S}_L \ast \mathcal{S}_{L^\ast} \rightarrow Sp_{2g}(k)$, and since $k$ is a field, the map $E$ is surjective. Indeed, it is well known that over a field the symplectic group is generated by elementary matrices of the shape $E(q)$ as above and by matrices of the form
\[
\left(
\begin{matrix}
	\alpha & 0 \\
	0& {}^t\alpha^{-1}
\end{matrix}
\right),
\] 
where $\alpha \in GL_g(k)$. But over a field  (cf. \cite[Thm. 66]{MR2001037})  every invertible matrix  is the product of two symmetric matrices, and if  $\alpha =p^{-1}q$ is such a factorization then the following two equalities in $Sp_{2g}(k)$
\[
\left(
\begin{matrix}
	1 & 0 \\
	q & 1
\end{matrix}
\right)
\left(
\begin{matrix}
	1 & -q^{-1} \\
	0 & 1
\end{matrix}
\right)
\left(
\begin{matrix}
	1 & 0 \\
	q & 1
\end{matrix} 
\right)
=
\left(
\begin{matrix}
	0 & -q^{-1} \\
	q & 0
\end{matrix}
\right) = m(q) \quad 
\]
and
\[
m(-p)m(q) = \left(
\begin{matrix}
	p^{-1}q & 0 \\
	0 & {}^t(p^{-1}q)^{-1}
\end{matrix}
\right),
\]
show that the elementary matrices suffice to generate the symplectic group.

In the symplectic space  $H(L)$, we have two preferred Lagrangians $L$ and $L^\ast$. For these we adopt the following convention: $L_n$ will denote $L$ if the integer $n$ is even and $L^\ast$ if it is odd. By definition, a reduced word in the free product $\mathcal{S}_L \ast \mathcal{S}_{L^\ast}$  is a sequence of symmetric linear maps $\underline{q}: (q_m,q_{m+1}, \dots,q_n)$ where $q_j: L_j \rightarrow L_{j+1}$. Such a sequence is said to be of \emph{type}  $(m,n)$ and is called a  \emph{Sturm sequence}. Barge-Lannes~\cite{MR2431916} associate to a Sturm sequence as before a Sylvester Matrix with support $L_m \oplus L_{m+1} \oplus \cdots \oplus L_n$ by the rule:
\[
S(\underline{q}) = \left(
\begin{matrix}
	(-1)^mq_m  & 1 & 0 & \cdots & 0  \\
	1 & (-1)^{m+1}q_{m+1} & 1 &   \ddots & \vdots  \\
	0 &  1 & \ddots & \ddots & 0\\
	\vdots  & \ddots  &\ddots & (-1) ^{n-1}q_{n-1}& 1 \\
	0 & \cdots & 0&  1 & (-1)^nq_n
\end{matrix}
\right)
\]

The sign that appears in this definition comes from our convention in defining the affine structure on the sets of Lagrangians transverse to $L_0$ or $L_1$, see~Section~\ref{subsec:comlagr}.

Let us now link the Sylvester matrix associated to a Lagrangian path and that associated to a Sturm sequence. 

\begin{proposition}[Prop. C.1~\cite{MR2431916}]\label{prop:c1}
	Let $m$ and $n$ be two integers, with  $m \leq n$ and let $\Lambda_{m-1}, \Lambda_m, \dots, \Lambda_n$ be a finite sequence of Lagrangians of the symplectic space $H(L)$ such that $\Lambda_{m-1} = L_{m-1}$. The following  conditions are equivalent:
	\begin{enumerate}
		\item[(i)] $\Lambda_{k-1} \pitchfork \Lambda_k$ for $m \leq k \leq n$
		\item[(ii)] There exists a Sturm sequence $(q_m,q_{m+1},\dots, q_n)$ on $L$ of type $(m,n)$ such that 
		\[
		\Lambda_k = E(q_m,q_{m+1},\dots,q_k) \cdot L_k \textrm{ for } m \leq k \leq n.
		\]
		
	\end{enumerate}
	Moreover, if condition  $(i)$ is satisfied, then the Sturm sequence that appears in  $(ii)$ is unique.
\end{proposition}
\begin{proof} The implication $ii)\Rightarrow i)$ is clear. 
	Conversely, we consider a Lagrangian path $\alpha: L_{m-1},\Lambda_m, \dots,\Lambda_n,\Lambda_{n}$. By definition $\Lambda_m \in \mathcal{L}_{\pitchfork L_{m-1}}$ and hence  there exists a unique $q_{m-1} \in \mathcal{S}_{L_{m-1}}$ such that $E(q_{m})L_m =\Lambda_m$. Assume by induction that we have constructed a unique  Sturm sequence $(q_m, \dots, q_k)$ for $k \geq n-1$ such that for all $s \leq k$ $\Lambda_s = E(q_m,q_{m-1}, \dots,q_s)L_s$. Let  $\Lambda'_{s+1} = E(q_m,q_{m-1}, \dots,q_s)^{-1} \Lambda_{s+1}$. By construction $\Lambda'_{s+1} \pitchfork L_s$ and hence, there exists a unique $q_{s+1} \in \mathcal{S}_{L_s}$ such that $\Lambda'_{s+1} = E(q_{s+1} L_{s+1})$, and by construction 
	\begin{align*}
		\Lambda_{s+1} & = E(q_m,q_{m-1}, \dots,q_s)\Lambda'_{s+1} \\
		& = E(q_m,q_{m-1}, \dots,q_s) E(q_{s+1})L_{s+1} \\
		& = E(q_m,q_{m-1}, \dots,q_s,q_{s+1})L_{s+1}
	\end{align*}
\end{proof}

Let $K = \ker (E: \mathcal{S}_L \ast \mathcal{S}_{L^\ast} \rightarrow  Sp_{2g}(k))$. By definition we have a   a short exact sequence of groups
\[
\xymatrix{
1 \ar[r] & K \ar[r] & \mathcal{S}_L \ast \mathcal{S}_{L^\ast} \ar[r] & Sp_{2g}(k) \ar[r] & 1.
}
\]

\subsection{Four natural functions on $K$}\label{subsec:4fonct}

Let $(m,n)$ denote one of the $4$ couples $(0,0), (0,1), (1,0), (1,1)$. Up to possibly adding to a word  $w \in  \mathcal{S}_L \ast \mathcal{S}_{L^\ast}$  either $\overline{0}$ or $\underline{0}$ at the beginning or at the end, we may assume that $w$ is of the type $(m,n)$, i.e it can be identified with a Sturm sequence of type $(m,n)$. Let us define:
\[
\begin{array}{rcl}
f_{m,n}: \mathcal{S}_L \ast \mathcal{S}_{L^\ast} & \longrightarrow & W(k) \\
w & \longmapsto & S(w), \textrm{ where  $w$ is of type } (m,n).
\end{array}
\]

\begin{proposition}\label{prop:4fonctdef}
The $4$ functions $f_{00}, f_{01}, f_{10}$ and $f_{11}$ are all well-defined.
\end{proposition}
\begin{proof} We only treat the case of the function $f_{00}$, the $3$ other cases can be treated in a similar way. It is enough to show that the value of   $f_{00}$ on $w$ is independent from the choice of the representative of type $00$ of $w$. There are only two types of ambiguity in the choice of the representative:
\begin{enumerate}
\item Given a representative  of type $00$ we may add to it the word  $\underline{0}\overline{0}$ at the beginning (resp. $\overline{0}\underline{0}$  at the end).
\item Inside a representative of type $00$ we may find a sub-word of the form
$
a \overline{0}b
$
with $a,b \in \mathcal{S}_{L}$ (resp. $a\underline{0}b$ with $a,b \in \mathcal{S}_{L^\ast}$ )
\end{enumerate} 

Case 1) Let $w: w_2,w_3,\dots, w_{2n}$ be a Sturm sequence of type $00$. We will only show the equality $f_{00}(\underline{0}\overline{0}w) = f_{00}(w)$, the other case is similar. Observe that 
\[
S(\underline{0}\overline{0}) = \left(\begin{matrix}
0 & 1 \\1 & 0
\end{matrix}\right)
\] 
is a non-degenerate and neutral form. The Lagrangian path  associated to $\underline{0}\overline{0}w$ is:
\[
L_0,L_1,E(\underline{0})L_0,E(\underline{0}\overline{0})L_1,E(\underline{0}\overline{0}w_2)L_0, \cdots, E(\underline{0}\overline{0}w)L_1.
\]
Since 
\[
L_1 =\left(\begin{matrix}
1 & 0 \\
w_2 & 0
\end{matrix}\right)L_1 \pitchfork \left(\begin{matrix}
1 & 0 \\
w_2 & 0
\end{matrix}\right)L_0 = E(\underline{0}\overline{0}w_2)L_0 
\]
the shorteneing Lemma~\ref{lem:raccourci}, tells us that
\[
S(\underline{0}\overline{0}w) = S(\alpha) \bot S(\beta), \quad (**)
\]
where $\alpha$ is the Lagrangian path
\[
L_1,E(\underline{0})L_0,E(\underline{0}\overline{0})L_1,E(\underline{0}\overline{0}w_2)L_0
\]
and $\beta$ is the Lagrangian path
\[
L_0,L_1,E(\underline{0}\overline{0}w_2)L_0, \cdots, E(\underline{0}\overline{0}w)L_1.
\]

Since $E(\underline{0})=E(\underline{0}\overline{0})= Id$ the path $\alpha$ is simply
\[
L_1,L_0,E(\overline{0})L_1,E(\overline{0}w_2)L_0
\]
which is associated to the Sturm sequence  $(\overline{0},w_2)$  and whose Sylvester matrix is neutral
\[
S(\alpha)= \left(\begin{matrix}
0 & 1 \\
1 & w_2
\end{matrix}\right).
\]

In $\beta$ one recognizes the Lagrangian path associated to the sequence $w$ and equality $(\ast\ast)$ in  $W(k)$ states that
\[
S(\underline{0}\overline{0}w_2\dots w_2n) = S(w_2 \dots w_{2n}).
\]

Case 2) Once more we only treat the first of the two sub-cases. The Sturm sequence associated to the representative of  type $00$ is:
\[
q_0, q_1, \dots, q_{2r}, \overline{0}, q_{2r+2}, \dots, q_{2n},
\]
and its associated Lagrangian path is
\[
L_0,L_1,E(q_0)L_0,\dots,E(q_0\cdots q_{2r-1})L_1, E(q_0\cdots q_{2r})L_0, E(q_0\cdots q_{2r}\overline{0})L_1,
\]
\[
E(q_0\cdots q_{2r}\overline{0}q_{2r+2})L_0,\dots,E(q_0\cdots q_{2r} \overline{0} q_{2r+2} \cdots q_{2n})L_0.
\]

Simplifying the evaluations, taking into account that $E$ is a group homomorphism and that $E(\overline{0})=Id$, this path is exactly
\[
L_0,L_1,E(q_0)L_0,\dots,{ \bf E(q_0\cdots q_{2r-1})L_1}, E(q_0\cdots q_{2r})L_0, E(q_0\cdots q_{2r})L_1,
\]
\[
{\bf E(q_0\cdots (q_{2r}+q_{2r+2}))L_0},\dots,E(q_0\cdots (q_{2r} + q_{2r+2}) \cdots q_{2n})L_0.\footnote{If $r=0$, by convention $E(q_0\cdots q_{2r-1})=Id$.}
\]

Now, the two Lagrangians in boldface characters are mutually transverse, for 
\[
L_1 = \left(\begin{matrix}
1 & 0 \\
q_{2r}+q_{2r+2} & 1
\end{matrix}\right)L_1 = E(q_{2r}+q_{2r+2})L_1 \pitchfork E(q_{2r}+q_{2r+2})L_0
\]
and hence $E(q_0\cdots q_{2r-1})L_1 \pitchfork E(q_0\cdots (q_{2r}+ q_{2r+2}))L_0$.

We  apply the Shortcut Lemma~\ref{lem:raccourci} to the two Lagrangians and observe that the Sylvester matrix of the initial sequence is therefore isometric to the direct sum of the two Sylvester matrices of the Lagrangian path
\[
L_0,L_1,E(q_0)L_0,\dots, E(q_0\cdots q_{2r-1})L_1,
\]
\[
 E(q_0\cdots (q_{2r}+q_{2r+2}))L_0,\dots,E(q_0\cdots (q_{2r} + q_{2r+2}) \cdots q_{2n})L_0,
\]
that is associated to the Sturm sequence
\[q_0,q_1, \dots,q_{2r}+q_{2r+2},\dots,q_{2n},\]
and of the path
\[
 E(q_0\cdots q_{2r-1})L_1, E(q_0\cdots q_{2r})L_0, E(q_0\cdots q_{2r})L_1,
 E(q_0\cdots (q_{2r}+q_{2r+2}))L_0.
\]

Since $E(q_{2r})L_1 = L_1$, this last path is the image by  $E(q_0\cdots q_{2r})$ of the path
\[
L_1, L_0, L_1,
 E(q_{2r+2})L_0.
\]
whose Sylvester matrix is
\[
\left(\begin{matrix}
0 & 1 \\
1 & q_{2r+2}
\end{matrix}\right)
\]
and is clearly neutral.

\end{proof}

Given a sub-word $k$ of $w$ which represents an element in the kernel $K$, the function $f_{00}$  behaves almost as a group homomorphism would do.

\begin{proposition}\label{prop:ssmotk}
Let $\underline{k} = k_0,\dots, k_{2k+1}  \in K$ be a word of type  $01$ and let  
\[
\underline{w} = q_0,\dots,q_{2r-1},k_0,\dots, k_{2\ell+1},q_{2\ell+2},\dots,q_{2n}
\]
be an arbitrary word in $\mathcal{S}_L\ast \mathcal{S}_{L^\ast}$ of type $00$ that contains $\underline{k}$ as a sub-word. Then:
\[
f_{00}(\underline{w}) = f_{01}(\underline{k}) + f_{00}(q_0, \dots, q_{2r-1},q_{2\ell+2}, \dots, q_{2n}).
\] 
\end{proposition}
\begin{proof}
	We write down the Lagrangian path associated to the Sturm sequence $\underline{w}$:
	\[
	L_0,L_1,E(q_0)L_0, \dots ,E(q_0 \cdots q_{2r-1})L_1, 
	\]
	\[
	E(q_0 \cdots q_{2r-1}k_0)L_0,\dots , E(q_0\cdots q_{2r-1}k_0 \cdots k_{2\ell +1})L_1,
	\]
	\[
	E(q_0\dots q_{2r-1}k_0 \cdots k_{2\ell +1}q_{2\ell +2})L_0, \dots, E(\underline{w})L_0.
	\]
Since the sequence $\underline{k}$ represents an element in $K$, the kernel of the evaluation map, we know that $E(\underline{k}) = Id$. In the Lagrangian path above, the last  Lagrangian on the second line and the first Lagrangian in the third line are by definition transverse:
\[
E(q_0\dots q_{2r-1}k_0 \cdots k_{2\ell +1})L_1
\pitchfork
E(q_0\dots q_{2r-1}k_0 \cdots k_{2\ell +1}q_{2\ell +2})L_0
\]

But $E(q_0\cdots q_{2r-1}k_0 \cdots k_{2\ell +1}) = E(q_0\cdots q_{2r-1})$, hence
\[
E(q_0\cdots q_{2r-1})L_1
\pitchfork
E(q_0\cdots q_{2r-1}k_0 \cdots k_{2\ell +1}q_{2\ell +2})L_0
\]

We can now apply the Shortcut Lemma~\ref{lem:raccourci} to the last Lagrangian of the first line and to the first Lagrangian on the third line; this tells us that  $f_{00}(\underline{w})$ is the sum of the classes of the Sylvester matrices associated to the two following Lagrangian paths. On the one hand
\[
L_0,L_1,E(q_0)L_0, \dots ,E(q_0 \cdots q_{2r-1})L_1, 
\]
\[
E(q_0\dots q_{2r-1}k_0 \cdots k_{2\ell +1}q_{2\ell +2})L_0, \dots, E(\underline{w})L_0.
\]
Since $E(\underline{k})=Id$, this is nothing else than the Lagrangian path
\[
L_0,L_1,E(q_0)L_0, \dots ,E(q_0 \cdots q_{2r-1})L_1, 
\]
\[
E(q_0\cdots q_{2r-1}q_{2\ell +2})L_0, \dots, E(q_0\cdots q_{2r-1}q_{2\ell +2}\cdots q_{2n})L_0
\]
whose class is by definition $f_{00}(q_0,\dots ,q_{2r-1},q_{2\ell +2},\dots, q_{2n})$,
and on the other hand the path
\[
E(q_0 \cdots q_{2r-1})L_1, 
E(q_0 \cdots q_{2r-1}k_0)L_0,\dots,
\]
\[
 E(q_0\cdots q_{2r-1}k_0 \cdots k_{2\ell +1})L_1,
E(q_0\dots q_{2r-1}k_0 \cdots k_{2\ell +1}q_{2\ell +2})L_0.
\]

By definition of the action of the symplectic group on Lagrangian paths, this path is the image by $E(q_0 \cdots q_{2r-1})$ of the Lagrangian path
\[
L_1, E(k_0)L_0,\dots,E(k_0 \cdots k_{2\ell +1})L_1,
E(k_0 \cdots k_{2\ell +1}q_{2\ell +2})L_0.
\]

Here again, as $E(k_0 \cdots k_{2\ell +1})=Id$, $E(k_0 \cdots k_{2\ell +1})L_1 = L_1$ and we recognize a Lagrangian loop with an extra term to its right:  $E(k_0 \cdots k_{2\ell +1}q_{2\ell +2})L_0 = E(q_{2\ell +2})L_0$. By Corollary~\ref{cor:les4chemins}, the Witt class of the Sylvester matrix associated to this path coincides with that of the Sylvester matrix associated to the path
\[
L_0,L_1, E(k_0)L_0,\dots,E(k_0 \cdots k_{2\ell +1})L_1
\]
which by definition is $f_{01}(\underline{k})$.

\end{proof}

\begin{remark}\label{rem:add0}
 If $\underline{k} \in K$ is a word of type $01$, by applying Corollary~\ref{cor:les4chemins} as in the last part of the above argument one can show that:
\[
f_{00}(\underline{k} \underline{0}) = f_{01}(\underline{k}) + f_{00}(\underline{0}) = f_{01}(\underline{k}),
\]
because in  $W(k)$, $f_{00}(\underline{0}) = 0$. 

In particular all  $4$ functions $f_{00}$, $f_{11}$, $f_{01}$ and $f_{01}$ coincide on $K$.
\end{remark}
 
 This leads us the key observation:
\begin{lemma}\label{lem:fsrukcentr}
The function $f_{01}:K \rightarrow W(k)$ is a group homomorphism invariant under the conjugation action of  $\mathcal{S}_L \ast\mathcal{S}_{L^\ast}$; it takes values in $I^2$, the square of the fundamental ideal.
\end{lemma}
\begin{proof} Let $\underline{k}$ and $\underline{\ell}$ denote two element is $K$, and fix for each of them a representative of type $01$, respectively $\underline{k}_{01}$ and   $\underline{\ell}_{01}$. Compute:
\begin{eqnarray*}
f_{01}(\underline{k}_{01}\underline{\ell}_{01}) & = & f_{00}(\underline{k}_{01}\underline{\ell}_{01}\underline{0}) \textrm{ by Remark~\ref{rem:add0}} \\
 &=& f_{01}(\underline{k}_{01}) + f_{00}(\underline{\ell}_{01}\underline{0}) \textrm{ by Proposition~\ref{prop:ssmotk} }\\
 & = & f_{01}(\underline{k}_{01}) + f_{01}(\underline{\ell}_{01}) \textrm{ by Remark~\ref{rem:add0}}
\end{eqnarray*}

To show invariance under conjugation, we fix a word  $w_{00} \in \mathcal{S}_{L} \ast \mathcal{S}_{L^\ast}$  of type $00$. Then $w_{00}\overline{0}k_{01}w_{00}^{-1}\overline{0}$  is a representative  of type $01$  of the conjugated of $\underline{k}$ by $\underline{w}$, and as before:
\begin{eqnarray*}
f_{01}(w_{00}\overline{0} k_{01}w_{00}^{-1}\overline{0}) & = & f_{00}(w_{00}\overline{0} k_{01}w_{00}^{-1}\overline{0}\underline{0})  \textrm{ by Remark~\ref{rem:add0}} \\
& = & f_{01}(k_{01}) +  f_{00}(w_{00}\overline{0}w_{00}^{-1}\overline{0}\underline{0}) \textrm{ by Proposition~\ref{prop:ssmotk} }
\end{eqnarray*}

But $w_{00}\overline{0}w_{00}^{-1}\overline{0}\underline{0}$ is a representative of type $00$ of $\underline{0}$, hence
 \[
f_{00}(w_{00}\overline{0}w_{00}^{-1}\overline{0}\underline{0}) = f_{00}(\underline{0}) = 0.
\]

We finally have to show that for all $\underline{k} \in K$ we have $f_{01}(\underline{k})  \in I^2$. Given a representative  $k_{01} = k_0,k_1\dots k_{2r+1}$ of type $01$ of $\underline{k}$, observe that the Sylvester matrix $S(k_{01})$  has as support the direct sum of  $r+1$ copies of the pair $L \oplus L^\ast$, in particular this vector space has even dimension and hence the class of  $S(k_{01})$ in $W(k)$ lies in $I =\ker( W(k) \longrightarrow \mathbb{Z}/2)$.

We need now to compute the  discriminant of $S(k_{01})$. By definition of $K$,

\[
E(k_0 k_1 \cdots k_{2r+1}) =Id \in GL_g(k) \subset Sp_{2g}(k).
\]

Now, if in a Strum sequence of even length, one multiplies the associated elementary matrices in the symplectic group, and as a result one obtains a matrix that is diagonal by blocks as:
\[
\left(
\begin{matrix}
a & 0 \\
0 & {}^ta^{-1}
\end{matrix}
\right) \in GL_g(k )\subset Sp_{2g}(k)
\]
Barge-Lannes~\cite[Scholie 5.5.6 p.111]{MR2431916} show then that the determinant of  $a$ is equal to the discriminant of the  Sylvester matrix associated to the initial Sturm sequence. In our present case, this gives us as we wanted:

\[
\operatorname{dis} (S(k_{01})) = \det Id = 1.
\]
\end{proof}

\subsection{A cocycle for the fundamental extension}

Let us push-out our original exact sequence:
\[
\xymatrix{
1 \ar[r] & K \ar[r]  & \mathcal{S}_L \ast \mathcal{S}_{L^\ast} \ar[r]    & Sp_{2g}(k)  \ar[r] & 1 \\
}
\]
along the composite morphism $f_{01} : K \rightarrow I ^2 \hookrightarrow W(k)$:
\[
\xymatrix{
1 \ar[r] & K \ar[r] \ar[d]^-{f_{01}} & \mathcal{S}_L \ast \mathcal{S}_{L^\ast} \ar[r]  \ar[dl]^{f_{00}} \ar[d] & Sp_{2g}(k) \ar@{=}[d] \ar[r] & 1 \\
0 \ar[r] & W(k) \ar[r]& \Gamma \ar[r] & Sp_{2g}(k) \ar[r] & 1
}
\]

We caution the reader that the diagonal arrow labeled $f_{00}$ is of course not a group homomorphism.  Invariance of $f_{01}$ with respect to the conjugation action implies that the bottom extension is a central extension. Finally, commutativity of the upper left triangle allows us to write an explicit $2$-cocycle for the bottom central extension:

\begin{propdef}\label{propdef:cocyclemu}
Let $x,y \in Sp_{2g}(k)$ and choose arbitrary type $00$  lifts of these elements, respectively denoted  $\tilde{x}$ and $\tilde{y}$, in $\mathcal{S}_L \ast \mathcal{S}_{L^\ast}$. Then the function:
\[
\mu(x,y) = f_{00}(\tilde{x}\overline{0}\tilde{y}) - f_{00}(\tilde{x}) - f_{00}(\tilde{y})
\]
is independent of the choice of the liftings and defines a $2$-cocycle for the extension
\[
\xymatrix@1{
0 \ar[r] & W(k) \ar[r]& \Gamma \ar[r] & Sp_{2g}(k) \ar[r] & 1.
}
\]
\end{propdef} 
\begin{proof}
Recall that $\Gamma$ as a set is a quotient of the set  $W(k)\times (\mathcal{S}_L \ast \mathcal{S}_{L^\ast})$. One checks from the definitions that the function which on a couple  $(w,\widetilde{x})$ takes the value
\[
r(w,\widetilde{x}) = w + f_{00}(\widetilde{x})
\]
is in fact a retraction of the extension. A direct and classical computation shows then that the $2$-cocycle associated to this retraction is $\mu$.
\end{proof}

We now identify the above cocyle with the Maslov index of a triple of Lagrangians introduced in Theorem~\ref{thm propMaslov}.

\begin{proposition}\label{prop:cocyis Maslov}
For any $x,y \in Sp_{2g}(k)$ we have:
\[
\mu(x,y) = \mu_{BL}(x^{-1}L,L,yL).
\]
\end{proposition} 
\begin{proof}
By definition $f_{00}(\tilde{x})$ is the Witt class of the Sylvester matrix associated to a precise path from $L_0$ to $E(\tilde{x})L_0$, say $\alpha_{\tilde{x}}$, and $f_{00}$ is 	the Sylvester matrix associated to a precise path from $L_0$ to $E(\tilde{y})L_0$, say $\alpha_{\tilde{y}}$. The problem is that $f_{00}(\tilde{x}\overline{0}\tilde{y})$ is not on the nose the Witt class of the Sylvester matrix associated to the concatenation $\alpha_{\tilde{x}} \ast \alpha_{\tilde{y}}$, instead it is the class of the concatenation $\alpha_{\tilde{x}} \ast E(\tilde{x}\overline{0})\alpha_{\tilde{y}}$.
By the equivariance of the Sylvester matrix of a Lagrangian path, Lemma~\ref{lem:sylvis equiv}, $f_{00}(\tilde{x}\overline{0}\tilde{y})$ is also the class of the Lagrangian path $E(\tilde{x})^{-1}(\alpha_{\tilde{x}}\ast E(\tilde{x}\overline{0})\alpha_{\tilde{y}} = (E(\tilde{x})^{-1}(\alpha_{\tilde{x}}))\ast \alpha_{\tilde{y}}$, and the same argument shows that $f_{00}(\tilde{x})$ is the class of the path $E(\tilde{x})^{-1} \alpha_{\tilde{x}}$, a path from $E(\tilde{x})^{-1}L$ to $L$. So the definition of $\mu(x,y)$ is precisely  that of $\mu(x^{-1}L,L,yL)$.
\end{proof}

From the general properties of the Maslov index of a triple of Lagrangians, we have the recover the following well-known properties of the Maslov cocycle.

\begin{proposition}\label{prop:propMaslovcocy}
The value of $\mu$ at $(x,y)$ only depends on the three Lagrangians  $L,x^{-1}L$ and $yL$. More generally, for any $\phi \in Sp_{2g}(k)$, the value of $\mu(x,y) \in W(k)$ only depends on the triple of Lagrangians $(\phi x^{-1}L, \phi L, \phi yL)$.  More precisely it only depends on the cosets determined by the elements $x^{-1}$ and $y$ in $Sp_{2g}(k)/Stab(L)$.
In particular if two of the three Lagrangians $L$, $xL$ and $xyL$ coincide, then $\mu(x,y)=0$.
\end{proposition}
\subsection{Triviality of $\mu$ modulo $I^2$ and computations}\label{subsec:trivmodI2}

Let us now consider the mod  $I^2$ reduction of the central extension defined by the $2$-cocycle $\mu$:

\[
\xymatrix{
1 \ar[r] & K \ar[r] \ar[d]^-{f_{01}} & \mathcal{S}_L \ast \mathcal{S}_{L^\ast} \ar[r]  \ar[d] & Sp_{2g}(k) \ar@{=}[d] \ar[r] & 1 \\
0 \ar[r] & W(k) \ar[r]\ar[d] & \Gamma \ar[r] \ar[d] & Sp_{2g}(k) \ar[r] \ar@{=}[d] & 1 \\
0 \ar[r] & W(k)/I^2 \ar[r]& \overline{\Gamma} \ar[r] & Sp_{2g}(k) \ar[r] & 1 
}
\]

As  $f_{01}: K \longrightarrow W(k)$ factors through $I^2$, the bottom extension trivially splits. Moreover, according to Proposition~\ref{prop:ssmotk}, the function $f_{00}: \mathcal{S}_{L}\ast \mathcal{S}_{L^\ast} \rightarrow W(k) \textrm{ mod } I^2$ factors through $Sp_{2g}(k)$, and the shape of $\mu$ given in  Definition-Proposition~\ref{propdef:cocyclemu} tells us that:

\begin{theorem}\label{thm:redmoI2esttriv}
The function $\Phi : Sp_{2g}(k) \rightarrow W(k)/I^2$ that associates to $x \in Sp_g(k)$ the element $f_{00}(x) \textrm{ mod } I^2$ is the unique function on $Sp_{2g}(k)$ that satisfies the equation:
\[
\forall x,y \in Sp_{2g}(k) \quad  \Phi(xy) - \Phi(x) - \Phi(y) = \mu(x,y) \textrm{ mod } I^2.
\]
\end{theorem}

We end by an observation on the function $\Phi$ that is both elementary and new. Recall that for any   $a \in k^\ast$  then   $\langle 1,-a\rangle$ denotes the associated Pfister form, and that the stabilizer of a Lagrangian $L$, $\textrm{Stab}_L \subset Sp_{2g}(k)$ is made of those matrices that according to the decomposition $k^{2g}= L \oplus L^\ast$ are of the form:
\[
\left(
\begin{matrix}
 x & u \\
 0 & {}^tx^{-1}
\end{matrix}
\right)
\]
where $x \in GL_g(k)$ and $u$ is a symmetric bilinear form on $L^\ast$. In particular there is a split short exact sequence:
\[
 1 \longrightarrow \cS_{L^\ast} \longrightarrow \textrm{Stab}_L \longrightarrow GL_g(k) \longrightarrow 1.
\]
A section is given by the function $h:x \longmapsto 
h(x)=\left(
\begin{matrix}
 x & 0 \\
 0 & {}^tx^{-1}
\end{matrix}
\right).
$
\begin{proposition}\label{prop phisurstabL}
The restriction of $\Phi$ to the subgroup $Stab_L$ is a morphism with values in the fundamental ideal $I$, it factors through the canonical projection  $\operatorname{Stab}_L \rightarrow  GL_g(k)$; more precisely: 
 \[\Phi(\left(
\begin{matrix}
 x & u \\
 0 & {}^tx^{-1}
\end{matrix}
\right)) =\langle 1,-\operatorname{det} x\rangle  \in I/I^2 = k^\ast/(k^\ast)^2.
\]
\end{proposition}
\begin{proof} That $\Phi\vert_{\textrm{Stab}_L}$ is a  morphism is a direct consequence of the properties of the cocycle $\mu$, see~Proposition~\ref{prop:propMaslovcocy}. The following relation among symplectic matrices:
\[
\left(
\begin{matrix}
 1& u{}^tx \\
 0 & 1
\end{matrix}
\right)\left(
\begin{matrix}
 x & 0 \\
 0 & {}^tx^{-1}
\end{matrix}
\right)=\left(
\begin{matrix}
 x & u \\
 0 & {}^tx^{-1}
\end{matrix}
\right),
\]
shows that to compute  $\Phi$ it is enough to consider the cases where $u=0$ on the one hand and  $x=Id$ on the other.

\begin{enumerate}
\item Let us start with   $u \neq 0$ and $x= Id$. Our matrix is therefore an element in  $\mathcal{S}_{L^\ast}$, and we can use the Sturm sequence  $\underline{0}, u,\underline{0}$ to compute the value of $\Phi$. 
The associated Sylvester matrix is:
\[
\left(
\begin{matrix}
0 & 1 & 0 \\
1 & -u & 1 \\
0 & 1 & 0
\end{matrix}
\right)
\]
An immediate computation shows that this quadratic form, whose support is  $L \oplus L^\ast \oplus L$ has as kernel $L$, embedded as the elements of the form $(x,0,-x)$. Its regularization has  therefore as support $L \oplus L^\ast \oplus L/L \simeq L\oplus L^\ast$. The regularized matrix of  $S(\underline{0}u\underline{0})$  is then:
\[
\left(
\begin{matrix}
0 & 1  \\
1 & -u  
\end{matrix}
\right)
\] 
which is neutral and therefore $\Phi(u) =0$.

 \item If $u=0$. Recall that since we work over a field, given   $x \in GL_g(k)$, there exist two symmetric forms $p$ et $q$ such that $x = p^{-1}q$.

Let us define
\[
m(q) = \left(
\begin{matrix}
 1 & 0 \\
 q & 1
\end{matrix}
\right)
\left(
\begin{matrix}
 1 & -q^{-1} \\
 0 & 1
\end{matrix}
\right)
\left(
\begin{matrix}
 1 & 0 \\
 q & 1
\end{matrix} 
\right).
\]

Then

\[
m(-p)m(q) = \left(
\begin{matrix}
 p^{-1}q & 0 \\
 0 & {}^t(p^{-1}q)^{-1}
\end{matrix}
\right).
\]

Let us  denote by $r$ the type $00$  Sturm sequence, \[\underline{-p},\overline{p^{-1}}, \underline{-p+q},\overline{-q^{-1}},\underline{q},
\]
where we underline elements in $\mathcal{S}_{L^\ast}$ and overline elements in $\mathcal{S}_L$.
To compute $\Phi(h(x))$ we use the extended sequence  $r\overline{0}\underline{0}$. Since $x$ stabilizes $L$, and $r\overline{0}$ is a preimage of $x$, Proposition~\ref{prop:propMaslovcocy} shows that
\begin{eqnarray*}
\Phi(h(x)) & = & f_{01}(r\overline{0}) + f_{00}(\underline{0})\textrm{ mod } I^2 \\
& = & f_{01}(r\overline{0})\textrm{ mod } I^2 
\end{eqnarray*}

Since $x$ also stabilizes $L^\ast$, $E(r\overline{0})L^\ast = L^\ast \pitchfork L$ and the Sylvester matrix of $r \overline{0}$ is non-degenerate. The support of this quadratic form is $L\oplus L^\ast  \oplus L\oplus L^\ast \oplus L\oplus L^\ast$ which has even dimension $6g$, and hence $\Phi(h(x) )\in I$.

We now have to compute  the discriminant of $S(r\overline{0})$.  As in the proof of Lemma~\ref{lem:fsrukcentr}, the evaluation of the components of the Sturm sequence  $r\overline{0}$ gives by construction:
\[
E(r\overline{0})= h(x)
\]
and as  $r\overline{0}$  has an even number of elements, Scholie~5.2.6 in \cite{MR2431916} tells us exactly that:
\[
\operatorname{disc}(S(r\overline{0}))= \det(x^{-1})
 = \det(x) \textrm{ mod } (k^\ast)^2.
 \]
\end{enumerate}
\end{proof}

Let us now consider the group  $Stab_{L\oplus L^\ast}$ of those elements that stabilize the \emph{decomposition} $L \oplus L^\ast$ of $H(L)$. The elements of this group fall into two  disjoint families:
\begin{enumerate}
\item Those elements of the form:
\[
h(x)= \left(
\begin{matrix}
 x & 0 \\
 0 & {}^tx^{-1}
\end{matrix}
\right)
\]
with $x \in GL_g(k)$, they stabilize both Lagrangians separately. 
\item Those elements of the form
\[
 m(y)=\left(
\begin{matrix}
 0 & -{}^ty^{-1} \\
 y & 0
\end{matrix}
\right)
\]
with $y \in GL_g(k)$, that swap the two Lagrangians  $L$ and $L^\ast$. 

These two types of elements together clearly generate the group $Stab_{L\oplus L^\ast}$.
\end{enumerate}

\begin{proposition}\label{prop:stabLetLast}
The restriction of the function $\Phi$ to the subgroup $Stab_{L\oplus L^\ast}$ is a group homomoprhism. Moreover
\[
\Phi( \left(
\begin{matrix}
 0 &  -{}^ty^{-1}\\
 y & 0
\end{matrix}
\right))
\]
is the Witt class of the pair  $\big((-1)^{\frac{g(g-1)}{2}}\operatorname{det}(y),3g\big) \in \big( k^\ast/(k^\ast)^2,\mathbb{Z}/4\big)$ via the morphism $F$ defined at the end of  Section~\ref{subsec:Witt}.
\end{proposition}
\begin{proof}
It is enough to show that $\forall v,w \in GL_g(k)$, $\Phi(m(v)m(w)) = \Phi(m(v)) + \Phi(m(w))$.

We apply the relation that characterizes $\Phi$ to $m(v)$ and $m(v^{-1})m(w) = h(v^{-1}w)$; by Proposition~\ref{prop:propMaslovcocy}:
\[
0 = \mu(m(v),m(v^{-1})m(w)) = \Phi(m(w)) - \Phi(m(v)) - \Phi(m(v^{-1})m(w)),
\]
which means that:
\[
\Phi(m(w)) - \Phi(m(v)) = \Phi(m(v^{-1})m(w))
\]
As this relation is satisfied for arbitrary $v$ and $w$, it is enough to show that $\Phi(m(v)) = - \Phi(m(v^{-1}))$. For this we compute $\Phi(m(v))$ from relation $(\ast)$. A representative of $m(v)$ is given by the Sturm sequence  $v:\underline{v}, \overline{-v^{-1}},\underline{v}$, whose associated Sylvester matrix is:
\[
S(v)=\left(
\begin{matrix}
v & 1 & 0 \\
1 & v^{-1} & 1 \\
0 & 1 & v
\end{matrix}
\right)
\]
since $E(v)L_0 = m(v)L_0 = L_1 \pitchfork L_0$, the Sylvester matrix $S(v)$ is non-degenerate, and its rank modulo  $4$ equals $g$. To compute the discriminant of $S(v)$, we calculate directly:
\[
\left| 
\begin{matrix}
v & 1 & 0 \\
1  & v^{-1} & 1 \\
0 & 1 & v
\end{matrix}
\right| = \left| 
\begin{matrix}
v & 1 & 0 \\
0 & 0 & 1 \\
0 & 1 & v
\end{matrix}
\right| = \left| 
\begin{matrix}
v & 1 & 0 \\
0  & -v^{-1} & 1 \\
0 & 0 & v
\end{matrix}
\right| = (-1)^{g}\operatorname{det}(v) 
\]
and hence:
\[
\operatorname{disc}(S(v)) = (-1)^{\frac{3g(3g-1)}{2}} (-1)^{g}\operatorname{det}(v) = (-1)^{\frac{g(g-1)}{2}} \operatorname{det}(v).
\]
\end{proof}

More generally, Proposition~\ref{prop:propMaslovcocy} and Proposition~\ref{prop:stabLetLast} show that the value of $\Phi$ on the matrices on the left side of the following equalities in the symplectic group only depends on $v$ and is computed using the two preceding Propositions:
\[
\left(
\begin{matrix}
v & x  \\
-{}^tx^{-1} & 0 
\end{matrix}
\right)  =\left( 
\begin{matrix}
x & v & \\
0  & -{}^tx^{-1}
\end{matrix}
\right)  \left( 
\begin{matrix}
 0& 1 \\
1  & 0
\end{matrix}
\right) 
\]
\[
\left(
\begin{matrix}
0 & -{}^tx^{-1}  \\
x & v
\end{matrix}
\right)  = \left( 
\begin{matrix}
 0& 1 \\
1  & 0
\end{matrix}
\right) \left( 
\begin{matrix}
x & v & \\
0  & -{}^tx^{-1}
\end{matrix}
\right) 
\]

Again by Proposition~\ref{prop:propMaslovcocy} and Proposition~\ref{prop phisurstabL} we can compute $\Phi$ on the matrix:
\[
\left( 
\begin{matrix}
x &  0 \\
hx  & -{}^tx^{-1}
\end{matrix}
\right) = \left(
\begin{matrix}
1 & 0  \\
h & 1
\end{matrix}
\right)  \left( 
\begin{matrix}
 x& 0 \\
0  & {}^tx^{-1}
\end{matrix}
\right) 
\]

More precisely:
\[
\Phi( \left( 
\begin{matrix}
 x& 0 \\
v & {}^tx^{-1}
\end{matrix}
\right) ) = \langle 1,-\det x\rangle + [vx^{-1}]
\]
where $[vx^{-1}]$ stands for the Witt class of the regularization of the quadratic form  $vx^{-1}$ with support $L$, and which is again symmetric.

\appendix

\bibliographystyle{alpha}

\end{document}